\documentclass[11pt]{article}

\usepackage[margin=.75in]{geometry}
\usepackage{multirow}
\usepackage{color}
\usepackage{amsmath, amsthm, amssymb,fancyhdr,mathrsfs, slashbox, enumerate}
\usepackage{graphicx}
\usepackage{latexsym}
\usepackage[round,authoryear]{natbib}

\newtheorem{theorem}{Theorem}
\newtheorem{lemma}{Lemma}

\newtheorem{example}{Example}

\newcommand{\mbold}[1]{\mbox{\boldmath $#1$}}
\newcommand{\balpha}{\boldsymbol{\alpha}}

\newcommand{\bd}{{\bf \Delta}}
\newcommand{\dd}{{\boldsymbol \delta}}
\newcommand{\ld}{\Delta}

\newcommand{\E}{\mathbb{E}}
\newcommand{\I}{\mathbb{I}}
\newcommand{\dt}{{\boldsymbol\theta}}
\newcommand{\BL}{\textcolor{black}}
\newcommand{\Ba}{\textcolor{black}}
\newcommand{\Bb}{\textcolor{black}}
\newcommand{\Bc}{\textcolor{black}}
\begin{document}

%%%%%%%%%%%%%%%%%%%%%%%%%%%%%%%%%%%%%%%%%%%%%%%%%%%%%%%%%%%%%%%%%%%%%%%%%%%%%%%%%%%%%
% Preliminaries
%%%%%%%%%%%%%%%%%%%%%%%%%%%%%%%%%%%%%%%%%%%%%%%%%%%%%%%%%%%%%%%%%%%%%%%%%%%%%%%%%%%%%
{\noindent Draft of  \today.}
\begin{center}
      \Large{\bf\sc Information content of partially rank-ordered set samples}
 %\footnote{draft of \today}}
\end{center}
\begin{center}
 \noindent{ {\sc Armin Hatefi}$^{\dagger}$ and
            {\sc Mohammad Jafari Jozani}}\footnote{\footnotesize{Corresponding author. \\ 
             Email: 
            m\_jafari\_jozani@umanitoba.ca.}}$^{,\ddagger}$  
            
\noindent{\footnotesize{\it  ${\dagger}$  The Fields Institute for Research in Mathematical Sciences \& Department of Statistical Sciences, University of Toronto.} }\\

\noindent{\footnotesize{\it ${\ddagger}$ Department of Statistics,  University of Manitoba, 
                        Winnipeg, MB, Canada, R3T 2N2. } }\\
\end{center}

%%%%%%%%%%%%%%%%%%%%%%%%%%%%%%%%%%%%%%%%%%%%%%%%%%%%%%%%%%%%%%%%%%%%%%%%%%%%%%%%%%%%%%%%%%
% Abstract
%%%%%%%%%%%%%%%%%%%%%%%%%%%%%%%%%%%%%%%%%%%%%%%%%%%%%%%%%%%%%%%%%%%%%%%%%%%%%%%%%%%%%%%%%%
\begin{center} {\small \bf Abstract}:  \end{center}

\footnotesize{
Partially rank-ordered set (PROS) sampling is a generalization of ranked set sampling  in which  rankers are not required to  fully rank the sampling units in each set, hence having more flexibility to perform  the necessary judgemental ranking process. \Bc{The PROS sampling} has a wide range of  applications in different fields  ranging from environmental and ecological studies to medical research and it has been shown to be superior over ranked set sampling and simple random sampling for estimating  the population mean. 
\Bc{%In this paper, 
We study the Fisher information content   and uncertainty structure  of the  PROS 
samples
and compare them with those of simple random sample 
(SRS) and  ranked set sample (RSS) counterparts of the same size from the underlying population.
We study
the uncertainty structure in terms of the Shannon entropy,  R\'enyi entropy 
and Kullback-Leibler (KL) discrimination measures. Several examples including the FI of PROS  samples from  the  location-scale family of distributions as well as a regression model are discussed. 
}
% To this end, we  first obtain   the Fisher information (FI)  content of  PROS samples and show that  it is more than its     simple random sample 
%(SRS) and  ranked set sample (RSS) counterparts of the same size from the underlying population.  In addition, the effect of subsetting errors  on the FI content of PROS  samples is explored. 
%Then, we study
%uncertainty structure of PROS samples in terms of the Shannon entropy, \BL{ R\'enyi} entropy 
%and Kullback-Leibler (KL) discrimination measures and compare them with the uncertainty measures associated with RSS and SRS data. 
}\\

%%%%%%%%%%%%%%%%%%%%%%%%%%%%%%%%%%%%%%%%%%%%%%%%%%%%%%%%%%%%%%%%%%%%%%%%%%%%%%%%%%%%%%%%%%
% Subjects
%%%%%%%%%%%%%%%%%%%%%%%%%%%%%%%%%%%%%%%%%%%%%%%%%%%%%%%%%%%%%%%%%%%%%%%%%%%%%%%%%%%%%%%%%%

%%%%%%%%%%%%%%%%%%%%%%%%%%%%%%%%%%%%%%%%%%%%%%%%%%%%%%%%%%%%%%%%%%%%%%%%%%%%%%%%%%%%%%%%%%
% Keywords and phrases
%%%%%%%%%%%%%%%%%%%%%%%%%%%%%%%%%%%%%%%%%%%%%%%%%%%%%%%%%%%%%%%%%%%%%%%%%%%%%%%%%%%%%%%%%%
\Bc{Fisher information, Shannon entropy, R\'enyi entropy, Kullback-Leibler information, 
Misplacement probability matrix.}

\normalsize

%%%%%%%%%%%%%%%%%%%%%%%%%%%%%%%%%%%%%%%%%%%%%%%%%%%%%%%%%%%%%%%%%%%%%%%%%%%%%%%%%%%%%%%%%%
% \section{ Introduction}
%%%%%%%%%%%%%%%%%%%%%%%%%%%%%%%%%%%%%%%%%%%%%%%%%%%%%%%%%%%%%%%%%%%%%%%%%%%%%%%%%%%%%%%%%%
\section{ Introduction } \label{sec:intro}

Ranked set   sampling is a 
powerful and  cost-effective  data collection  technique which can be used  to obtain more representative samples from the underlying population when  a small number of sampling units can be fairly accurately ordered with respect to a variable of interest without actual measurements on them and at little cost. It is assumed that the exact measurement of the variable of interest is very costly  but  ranking sampling units  is cheap. 
Ranked set sampling has many applications  in  industrial statistics, 
environmental and ecological studies as well as  medical research.  Some recent examples include estimating phytomass \citep{muttlak}, stream habitat area \citep{mode},  mean and variance in flock management \citep{ozturk-bilgin}
  as well as studying 
the association between smoking exposure and three important carcinogenic biomarkers 
in a lung cancer decease study  \citep{Chen-Wang-04-Biom} and  in a fishery research  for 
estimating the mean stock abundance using the catch-rate data available from previous 
years as a concomitant variable \citep{Wang}. For recent  overviews of the theory and applications of  ranked set sampling and some of its variations  see \citet{Wolfe-12} and \cite{chen-bai-sinha}.
 
To obtain a ranked set sample (RSS),  an initial simple random sample (SRS) of size $k$ is taken. These units are ordered, but without actually being measured; we call this judgement ranking, which may be perfect or imperfect. Upon ranking, only the smallest unit is measured. Following this, a second SRS of size $k$ is taken, ranked and the second smallest unit  is measured. This process is repeated until the largest unit in a SRS of size  $k$  has been measured.   
In this process, the ranker is asked to declare unique ranks  for each  unit inside 
the sets. There are many situations  where it is difficult to rank all of the sampling units in a set with high confidence, particularly when subjective information is utilized in the ranking process.   Forcing rankers to declare unique ranks can  lead to inflated within-set judgment ranking 
error and consequently  to 
invalid statistical inference. 
Partially rank-ordered set (PROS) sampling design  is a  generalization of ranked set sampling, due to \cite{Oztu-11-ees}, which is aimed at reducing the  impact of ranking error and the burden on  rankers by not requiring them to provide  a full ranking of all the units in each set. Under PROS sampling technique,  rankers  have more flexibility by being able to  
divide the sampling units  into  subsets of pre-specified sizes. 
These subsets are 
partially rank-ordered so that each  unit in subset $h$ has a rank smaller  than the ranks of   units in  subset $h^{'}$  for all $ h^{'}\geq h$. 
An observation is then collected from one of these 
subsets in each set.   \cite{Hate-Joza-Oztu-13} used  PROS sampling design   to estimate  the parameters of a finite mixture model  to analyze the  age structure of a fish species.  
\cite{frey} studied nonparametric mean estimation using PROS sampling design. \cite{CJS:CJS11167}  proposed statistical procedures that utilize PROS  data from multiple observers to assist in the selection of units for measurement in a basic ranked set sample design or to construct a judgment post-stratified design. 

In this paper,  we study    information  and uncertainty content of     PROS 
samples. 
To this end, in Section \ref{se:pros-sam},  we provide a formal description of PROS sampling and present  some preliminary results on  distributional   properties  of PROS samples. 
 In Section \ref{sec:fi-pros}, we obtain   the Fisher information  (FI) content of  PROS samples and  show that it is  more than the FI content of its SRS and RSS counterparts of the same size.   Several examples including the FI of PROS  samples from  the  location-scale family of distributions as well as  a simple linear regression model are also discussed in this section. In addition, the effect of subsetting errors when applying PROS sampling design on the FI content of  samples is explored.  In Section \ref{sec:oic}, we study  information and uncertainty of  PROS samples   using the Shannon 
entropy, \BL{  R\'enyi} entropy and KL information measures and compare them with their SRS and RSS counterparts. Finally, in Section \ref{concluding}, we give some concluding remarks.

%%%%%%%%%%%%%%%%%%%%%%%%%%%%%%%%%%%%%%%%%%%%%%%%%%%%%%%%%%%%%%%%%%%%%%%%%%%%%%%%%%%%%%%%%%%%%
%% \section{PROS sampling}
%%%%%%%%%%%%%%%%%%%%%%%%%%%%%%%%%%%%%%%%%%%%%%%%%%%%%%%%%%%%%%%%%%%%%%%%%%%%%%%%%%%%%%%%%%%%%
\section{Preliminary results on distributional properties of PROS samples}\label{se:pros-sam}

%
%In ranked set sampling, rankers have to assign unique ranks to all the 
%units in each set even if they are  not confident 
%enough to rank the units accurately. This imposed ranking  forces the rankers to assign a rank at 
%random in the sets; therefore  inflating  within-set judgment ranking error. This could result in invalid statistical inference including tests with inflated type I error rates, 
%confidence intervals with incorrect coverage probabilities and biased estimates.   
%\citet{Oztu-11-ees} introduced PROS sampling so that rankers do not 
%need  to assign unique ranks to the units among which the ranking is not easy;
%in other words, the 
%rankers are asked to put these units into partially rank-ordered subsets.

To obtain a PROS sample of size $n$,  we   choose a set size $S$ and a design parameter  $D=\{d_1, \ldots, d_n\}$ that partitions the set  $\{1,\ldots,S\}$ into $n$
mutually exclusive subsets. First,  $S$ units are randomly selected 
and are assigned into subsets $d_r,r=1,\ldots,n$,    without actual 
measurement of the variable of interest and only based on visual inspection or judgment,  etc.
Then a unit is selected at random for measurement from the subset $d_1$ and it  is denoted 
by \BL{ $X_{(d_1)1}$}. Selecting another $S$ units assigning them 
into subsets, a unit is randomly drawn from subset $d_2$ and then  it is quantified and  denoted by $X_{(d_2)1}$. This process is repeated until we randomly draw a unit \BL{  from}   $d_n$ resulting in $X_{(d_n)1}$.
This constitutes  one  \textit{cycle} of PROS sampling technique. The cycle is then 
repeated $N$ times to generate a PROS sample of the size $Nn$, i.e. $\{X_{(d_r)i};r=1,\ldots,n;i=1,\ldots,N\}$.
 Table \ref{ta:pros-ex} shows the construction of a balanced PROS sample with  \Ba{$S=6,n=2, N=2$} and the  design parameter
$D=\{d_1,d_2\}=\{\{1,2,3\},\{4,5,6\}\}$. Each set includes  
six units assigned into two partially ordered subsets. This partial ordering provides
the information that the units in $d_1$ have the smaller ranks than units in $d_2$. In this subsetting process we do  not assign any ranks to units 
within each  subset  so that these units  are equally likely to take any place in the
subset.  One unit, in each set from the bold faced subset, is randomly drawn and is quantified. The fully measured 
units are denoted by $X_{(d_r)i}$, $r=1, 2$; $ i=1, 2$. 
\begin{table}[h!]
\begin{center}
\caption{\small An example of PROS design}
\small{\begin{tabular}{cccc} \hline\hline
cycle & set &  Subsets & Observation \\ \hline
 1    & $S_1$ & $D_1=\{\mbold{d_1},d_2\}=\{ \mbold{\{1,2,3\}},\{4,5,6\}\}$ & $X_{(d_1)1}$  \\ 
      & $S_2$ & $D_2=\{d_1,\mbold{d_2}\}=\{ \{1,2,3\},\mbold{\{4,5,6\}} \}$ & $X_{(d_2)1}$  \\\hline
     2    & $S_1$ & $D_1=\{\mbold{d_1},d_2\}=\{ \mbold{\{1,2,3\}},\{4,5,6\} \}$ & $X_{(d_1)2}$  \\
      & $S_2$ & $D_2=\{d_1,\mbold{d_2}\}=\{ \{1,2,3\},\mbold{\{4,5,6\}} \}$ & $X_{(d_2)2}$  \\\hline\hline
 \end{tabular}}
  \label{ta:pros-ex}
 \end{center}  
 \end{table}    
 
Throughout the paper,  without loss of generality, we assume that $N=1$ (unless otherwise specified) and we use \BL{  PROS($n, S)$} to denote a PROS sampling  design  with the set size $S$, the number of subsets $n$ and the design parameter $D=\{ d_r, r=1, \ldots, n \}$ where   $d_r=\{(r-1)m+1,\ldots, rm\}$,  in which  $m=S/n$ is the number of  
 unranked observations in each  subset.
 We note that RSS and SRS can be expressed as special cases of the \BL{  PROS($n, S)$}
 design  when $S=n$ and $S=1$, respectively.
 
Suppose  $X$ is  a continuous random variable with probability density function (pdf) $f(x;\dt)$
and cumulative density function (cdf) $F(x;\dt)$, where  $\dt$ is  the vector  of unknown parameters with   $\dt \in \mathbb{R}^{p}$. \BL{ 
Let ${\bf X}_{pros}=\{ X_{(d_r)}, r=1, \ldots, n\}$  be a perfect PROS($n, S$) sample of size $n$ from $f(\cdot, \dt)$}. The PROS data  likelihood function of \BL{  $\dt$} is given by the joint pdf of ${\bf X}_{pros}$ as follows:
\begin{eqnarray*}\label{l-pros-in}
L(\dt|{\bf x}_{pros})= f({\bf x}_{pros}; \dt)= \prod_{r=1}^{n}
\left\{\frac{1}{m} \sum_{u \in d_r}  f^{(u:S)}({x_{(d_r)}};\dt)\right\}, 
\end{eqnarray*} 
where $f^{(u:S)}(\cdot;\dt)$ is the pdf of the $u$-th order statistic of a SRS of size $S$ from $f(\cdot; \dt)$. 
For each  $X_{(d_r)}$  define the  latent vector  
$\bd^{(d_r)}= \left(\ld^{(d_r)}(u), u\in \Bc{d_r= \{(r-1)m+ 1, \ldots, rm\}}\right)$,   where 
\[  
\ld^{(d_r)}(u) = 
\left\{
       \begin{array}{ll}
       1 & \mbox{if  $X_{(d_r)}$ is selected from the $u$-th position within the subset $d_r$}; \\
       0 & \mbox{otherwise},
       \end{array} 
\right. 
\]
with  $\sum_{u \in d_r} \ld^{(d_r)}(u)=1$. Denote 
${\bf Y}_{pros}= \{ (X_{(d_r)},\bd^{(d_r)}),r=1,\ldots,n \} $ as the complete PROS data  consisting of  $X_{(d_r)}$ and their corresponding latent vectors $\bd^{(d_r)}$, $r=1,\ldots,n$.
 The  complete PROS  data likelihood function of $\dt$  using  the joint pdf of ${\bf Y}_{pros}$ is  given by 
\begin{eqnarray}\label{l-pros}
L(\dt|{\bf y}_{pros})=  f({\bf y}_{pros}; \dt)= \prod_{r=1}^{n}\prod_{u \in d_r} 
\left\{\frac{1}{m} f^{(u:S)}({x_{(d_r)}};\dt)\right\}^{\dd^{(d_r)}(u)}.
\end{eqnarray}
%
%According to the fact that  $X_{(d_r)}$ is uniformly distributed in the subset $d_r$;
%the joint distribution of $(X_{(d_r)},\bd^{(d_r)})$ is given by 
%\begin{eqnarray}\label{f-per-x-p}
%      f(x_{(d_r)},\dd^{(d_r)};\dt)=
%      \prod_{u \in d_r} \left\{\frac{1}{m} f^{(u:S)}({x_{(d_r)}};\dt)\right\}^{\dd^{(d_r)}(u)},
%\end{eqnarray}
Furthermore,  by summing the  joint distribution of $(X_{(d_r)},\bd^{(d_r)})$ over  $\bd^{(d_r)}=\dd^{(d_r)}$,    the marginal distribution of $X_{(d_r)}$ is obtained  as follows
\begin{align}\label{f-per-x}
      f_{(d_r)}(x_{(d_r)};\dt) 
      =   \sum_{\dd^{(d_r)}} 
            f(x_{(d_r)},\dd^{(d_r)};\dt) 
%      &=   \sum_{\dd^{(d_r)}} \prod_{u \in d_r} 
%            \left\{\frac{1}{m} f^{(u:S)}({x_{(d_r)}};\dt)\right\}^{\dd^{(d_r)}(u)} \nonumber    \\
      =   \frac{1}{m} \sum_{u \in d_r} 
            f^{(u:S)}({x_{(d_r)}};\dt).
\end{align}
Also, one can easily check that 
\begin{align}\label{f-x}
\frac{1}{n}\sum_{r=1}^{n}f_{(d_r)}(x;\dt)
%=   \frac{1}{S} \sum_{r=1}^{n} \sum_{u \in d_r} 
%            f^{(u:S)}(x;\dt)
%=   \frac{1}{S} \sum_{v=1}^{S}  
%            f^{(v:S)}(x;\dt)
=   f(x;\dt).
\end{align}
In addition,  the conditional distribution of $\bd^{(d_r)}$ given $X_{(d_r)}$ is 
\begin{align}\label{f-per-p|x}
      f(\dd^{(d_r)}\big|x_{(d_r)};\dt)=
       \prod_{u \in d_h} 
      \left\{
      \frac{ f^{(u:S)}({x_{(d_r)}};\dt)}
      {\displaystyle\sum_{\BL{ v} \in d_r} 
       f^{(\BL{ v}:S)}({x_{(d_r)}};\dt) }
      \right\}^{\dd^{(d_r)}(u)}.
\end{align}
%Let ${\bf X}_{pros}=\{ X_{(d_r)}, r=1, \ldots, n\}$ 
%denote the incomplete PROS data and  ${\bf Y}_{pros}= \{ (X_{(d_r)},\bd^{(d_r)}),r=1,\ldots,n \} $ denote the complete PROS data  consisting of  $X_{(d_r)}$ and their corresponding latent vectors $\bd^{(d_r)}$, $r=1,\ldots,n$.
%Under PROS sampling with set size $S$ and number of subsets $n$, the incomplete data likelihood function is
%\begin{eqnarray*}\label{l-pros-in}
%L(\dt|{\bf x}_{pros})= \prod_{r=1}^{n}
%\left\{\frac{1}{m} \sum_{u \in d_r}  f^{(u:S)}({x_{(d_r)}};\dt)\right\},
%\end{eqnarray*} 
%% According to the fact that
%%the $X_{(d_r)}$s are independent random variables distributed as \eqref{f-per-x} $r=1,\ldots,n$ respectively;
%while the complete data likelihood function is given by
%\begin{eqnarray}\label{l-pros}
%L(\dt|{\bf y}_{pros})= \prod_{r=1}^{n}\prod_{u \in d_r} 
%\left\{\frac{1}{m} f^{(u:S)}({x_{(d_r)}};\dt)\right\}^{\dd^{(d_r)}(u)}.
%\end{eqnarray}

%%%%%%%%%%%%%%%%%%%%%%%%%%%%%%%%%%%%%%%%%%%%%%%%%%%%%%%%%%%%%%%%%%%%%%%%%%%%%%%%%%%%%%%%%%%%%%
%% \section{Fisher Information in PROS design}
%%%%%%%%%%%%%%%%%%%%%%%%%%%%%%%%%%%%%%%%%%%%%%%%%%%%%%%%%%%%%%%%%%%%%%%%%%%%%%%%%%%%%%%%%%%%%%
\section{FI content of  PROS samples }\label{sec:fi-pros}

In this section,  we  first obtain the FI content of  ${\bf Y}_{pros}$,  the complete PROS data,  and derive analytic results to compare it  with   the FI content of 
SRS and RSS data  of the same size. We give examples regarding   the  location-scale
family of distributions as well as a simple linear regression model. Then,  we study the FI content of  ${\bf X}_{pros}$ by modelling an imperfect PROS design involving misplacement errors in the subsetting process.    The FI of PROS samples can play a key role in its  theory and application  to study  the asymptotic behaviour of  the  maximum likelihood estimators of $\dt$  as well as the derivation of the  Cramer-Rao 
lower bound for unbiased estimators of $\dt$ or some of its functions   based on PROS samples.

  Under the usual regularity conditions \BL{  (e.g., \citealp{chen-bai-sinha})}, the FI matrix is calculated by 
$
\I(\dt)=-\E[D_{\dt}^2\log f(X;\dt)],
$
provided the expectation exists, where $D_{\dt}^l$ refers to the $l$-th derivatives  of the 
log-likelihood function with respect to $\dt$ with $D_{\dt}^1=D_{\dt}$. For any two matrices  $A$ and $B$ of the same size,  we use $A\ge0$ and $A\ge B$  
 to  indicate that $A$ and $A-B$ are \BL{  non-negative definite} matrices. We also let   $ \phi_u(\lambda)=(u-1)\, I(\lambda=0)+(S-u)\,I(\lambda=1)$ with $\lambda\in\{0,1\}, u=1, \ldots, S$,  
  where $I$ is the usual indicator function.
 
% 
% the following notations in this section. 
%  \begin{notation}
%  \label{note2}
% Let $A$ and $B$ be two matrices of the same dimensions. We define $A\ge0$ and $A\ge B$ to 
% indicate that $A$ and $A-B$ are nonnegative matrices.
% \end{notation}
% \begin{notation}\label{note3}
%  We let 
%  $ \phi_u(\lambda)=(u-1)\, I(\lambda=0)+(S-u)\,I(\lambda=1)$ with $\lambda\in\{0,1\},u=1,\ldots,S$,  
%  where $I$ is the usual indicator function.
% \end{notation}

%%%%%%%%%%%%%%%%%%%%%%%%%%%%%%%%%%%%%%%%%%%%%%%%%%%%%%%%%%%%%%%%%%%%%%%%%%%%%%%%%%%%%%%%%%
% \section{Fisher Information Criterion}
%%%%%%%%%%%%%%%%%%%%%%%%%%%%%%%%%%%%%%%%%%%%%%%%%%%%%%%%%%%%%%%%%%%%%%%%%%%%%%%%%%%%%%%%%%
\subsection{ FI matrix of complete PROS data ${\bf Y}_{pros}$} \label{subsec:fis}
 
\BL{ 
Here we obtain the FI matrix of ${\bf Y}_{pros}$ under perfect subsetting assumption. To do so, we need  the following useful result.
}

\begin{lemma}\label{le:chen-pros}
Suppose  $Y_r=X_{(d_r)}$,  with pdf $f_{(d_r)}(\cdot;\dt)$,  is observed from  a continuous distribution with pdf $\Bc{f(\cdot; \dt)}$ and cdf $F(\cdot; \dt)$, respectively,  using a  \BL{ { \rm PROS($n, S$)}}  design.  Let $\dd^{(d_r)}(u)$ be  the 
 latent variable associated with $X_{(d_r)}$. For 
 any $\lambda\in\{0,1\}$ and any function $G(\cdot)$,
\begin{eqnarray*}
\E\left\{
\sum_{r=1}^{n} \sum_{u \in d_r}  \frac{ \phi_u(\lambda) \,\dd^{(d_r)}(u)\, G(Y_r)}{\lambda+(1-2\lambda)\, F(Y_r;\dt)}
\right\}
%\E\left\{
%\sum_{r=1}^{n} \sum_{u \in d_r} (S-u) \frac{\dd^{(d_r)}(u)\, G(Y_r)}{{\bar F}(Y_r;\dt)}
%\right\}
= n(S-1) \E[G(X)],
\end{eqnarray*}
subject to the existence of the expectations.
\end{lemma}

\begin{proof}  
Let  $\lambda=0$.  By the  total law of expectations and  equation  \eqref{f-per-p|x} we get
\begin{align*}
\E
\left\{
\sum_{r=1}^{n} \sum_{u \in d_r} (u-1) \frac{\dd^{(d_r)}(u) \, G(Y_r)}{F(Y_r;\dt)}
\right\}
%&=&
%\frac{1}{m}\sum_{r=1}^{n} \sum_{u \in d_r} (u-1) \E 
%\left\{
%\frac{G(Y_r)}{F(Y_r;\dt)} \frac{f^{(u:S)}(Y_r;\dt)}{f_{(d_r)}(Y_r;\dt)}
%\right\} \\
&= \frac{1}{m}\sum_{r=1}^{n}\sum_{u \in d_r} (u-1) \int \frac{G(x)}{F(x;\dt)} f^{(u:S)}(x;\dt) dx\\
%&=&  \frac{1}{m}\sum_{r=1}^{n}\sum_{u \in d_r} (u-1) \int G(x) f(x;\dt) 
%\left\{
%S {{S-1}\choose{r-1}} [F(x;\dt)]^{u-2} [{\bar F}(x;\dt)]^{S-u}
%\right\} dx \\
&=  \frac{S}{m} \int G(x)  f(x;\dt) 
\left\{ \sum_{v=1}^{S} (v-1)
 {{S-1}\choose{v-1}} [F(x;\dt)]^{v-2} [{\bar F}(x;\dt)]^{S-v}
\right\} dx \\
%&=&  n(S-1) \int G(x)  f(x;\dt) 
%\left\{
% \sum_{v=2}^{S} {{S-2}\choose{v-2}} [F(x;\dt)]^{v-2} [{\bar F}(x;\dt)]^{S-v}
%\right\} dx \\
&= n(S-1) \E [G(X)],
\end{align*}
The proof for $\lambda=1$ is similar and \Bc {hence is omitted}. 
\end{proof}

\noindent  Now, we obtain the FI content of ${\bf Y}_{pros}$ and compare it with its  SRS counterpart  of the same size.  
\begin{theorem}\label{th:fi-pros-srs}
Under  the usual regularity conditions \BL{  {\rm (e.g., \citealp{chen-bai-sinha})}}, the  FI matrix  of a complete \BL{ {\rm PROS($n, S$)}} sample  of size $n$  from $f(\cdot; \dt)$  is given by 
\[
\I_{pros}(\dt)=\I_{srs}(\dt)+ \mathbb{K}(\dt),
\]
%where as usual
%\[
%I_{SRS}(\dt)=-n \E\left\{D_{\dt}^2 \log f(X;\dt)\right\},
%\]
where  $\I_{srs}(\dt)$ denotes 
the FI matrix of  a SRS of size $n$,  
\[
\mathbb{K}(\dt)=n(S-1) 
\E\left\{
\frac{[D_{\dt}F(X;\dt)][D_{\dt}F(X;\dt)]^{\top}}{F(X;\dt) {\bar F}(X;\dt)}
\right\},
\]
is a  non-negative definite matrix
and the expectation is taken with respect to $X$. 
\end{theorem}
\begin{proof}
Let $Y_r=X_{(d_r)}, r=1,\ldots,n$. Using \eqref{l-pros}, the log-likelihood function  of $\dt$
 can be written as
\begin{eqnarray*}
\BL{ l_{pros}(\dt) = cst + l_{srs}^* (\dt) + \Gamma_p (\dt),}
\end{eqnarray*}
\BL{  where $cst=n\log \{n \binom{S-1}{n-1} \}$  is a constant with respect to $\dt$ and}
\begin{eqnarray*}
\Gamma_p(\dt)=\sum_{r=1}^n \sum_{u \in d_r} \sum_{\lambda=0}^{1}
\phi_u (\lambda)\,\dd^{(d_r)}(u)  
\log[ \lambda+(1- \Bb{2} \lambda)\,F(y_r;  \dt)] , 
\end{eqnarray*}
%and also
%\[
%l_{P, SRS} (\dt) =  \sum_{r=1}^{n} \log f(x_{(d_r)};\dt).
%\]
and
%\begin{eqnarray*}\label{fi-srs}
$-\E[D_{\dt}^2 l_{srs}^*(\dt)]= \I_{srs}(\dt)$.
%\end{eqnarray*}
Taking second  derivatives of   $\Gamma_p(\dt)$ with respect to $\dt$, one gets 
\begin{eqnarray}\label{d2-gamma}\nonumber
D^2_{\dt}\Gamma_p(\dt) 
= 
\sum_{r=1}^n \sum_{u\in d_r} \sum_{\lambda=0}^{1}  \phi_u(\lambda)
\dd^{(d_r)}(u) \left\{ 
\frac{ \Bb{(-1)^{\lambda}} D^2_{\dt} F(y_r;  \dt)}{\lambda+(1- \Bb{2} \lambda)\,F(y_r;  \dt) } 
-\frac{[D_{\dt} F(y_r;  \dt)] [D_{\dt} F(y_r;  \dt)]^{\top}}{[\lambda+(1- \Bb{2} \lambda)\,F(y_r;  \dt)]^2} 
\right\}.
%&&-  \sum_{r=1}^n \sum_{u\in d_r}  (S-u) \dd^{(d_r)}(u) \left\{
%\frac{D^2_{\dt} F(x_{(d_r)};  \dt)}{{\bar F}(x_{(d_r)};  \dt)}
%-\frac{[D_{\dt} F(x_{(d_r)};  \dt)] [D_{\dt} F(x_{(d_r)};  \dt)]^t}{[{\bar F}(x_{(d_r)};  \dt)]^2}
% \right\}.
\end{eqnarray}
%
%\begin{lemma}\label{le:de-w}
%\begin{eqnarray*}
%\E
%\left\{
%\sum_{r=1}^{n} \sum_{u \in d_r} (u-1) \dd_i^{(d_r)}(u) G(X_{(d_r)i})
%\right\}= 
%S(S-1)\E
%\left\{
%G(X_i) F(X_i)
%\right\}
%\end{eqnarray*}
%\end{lemma}
\Bb{ Using Lemma \ref{le:chen-pros}, 
%with  
%$
%G(x)=D^2_{\dt}F(x;\dt),
%$
we have
\begin{align}\label{e-gamma1}
\E
\left\{
\sum_{r=1}^{n} \sum_{u \in d_r} (u-1)  
\frac{\dd^{(d_r)} (u)\,D^2_{\dt}F(Y_r;\dt)}{F(Y_r;\dt)}
\right\}&=
\E
\left\{
\sum_{r=1}^{n} \sum_{u \in d_r} (S-u) 
\frac{\dd^{(d_r)} (u) \,D^2_{\dt}F(Y_r;\dt)}{{\bar F}(Y_r;\dt)}
\right\} \\\nonumber
&= S(S-1) 
\E\left\{
D^2_{\dt} F(X;\dt)
\right\}.
\end{align}}
\Bb{ Similarly,
%, by choosing $G(x)=
%\frac{[D_{\dt}F(x;\dt)][D_{\dt}F(x;\dt)]^{\top}}{\lambda+(1-2\lambda)F(x;\dt)}$
 by Lemma \ref{le:chen-pros}, we obtain
\begin{align}\label{e-gamma2}\nonumber
&\E
\left\{
\sum_{r=1}^{n} \sum_{u \in d_r} 
\phi_{u}(\lambda)\,\dd^{(d_r)} (u) 
\times \frac{[D_{\dt}F(Y_r;\dt)][D_{\dt}F(Y_r;\dt)]^{\top}}{[\lambda+(1-2\lambda)F(Y_r;\dt)]^2} \right\}
\\ &=
n(S-1) 
\E
\left\{
\frac{[D_{\dt}F(X;\dt)][D_{\dt}F(X;\dt)]^{\top}}{\lambda+(1-2\lambda)F(X;\dt)}
\right\}, \quad \lambda\in\{0,1\}.
%\left\{
%\frac{[D_{\dt}F(X;\dt)][D_{\dt}F(X;\dt)]^{T}}{F(X;\dt)}
%\right\},
\end{align}}
Taking expectation of $D^2_{\dt}\Gamma_P(\dt)$ and from \eqref{e-gamma1} and \eqref{e-gamma2}, we obtain 
\begin{eqnarray}\label{fi-gamma}
\mathbb{K}(\dt)=-\E [D^2_{\dt} \Gamma_p(\dt)]= 
n(S-1) 
\E\left\{
\frac{[D_{\dt}F(X;\dt)][D_{\dt}F(X;\dt)]^{\top}}{F(X;\dt){\bar F}(X;\dt)}
\right\},
\end{eqnarray}
which completes the proof. 
\end{proof}
%Eventually from \eqref{fi-gamma}, the FI is shown to be 
%\begin{eqnarray*}\label{fi-pros}
%\I_{pros}(\dt)=  \I_{srs}(\dt) 
%+ n(S-1) 
%\E\left\{
%\frac{[D_{\dt}F(X;\dt)][D_{\dt}F(X;\dt)]^{\top}}{F(X;\dt){\bar F}(X;\dt)}
%\right\}. 
%\end{eqnarray*}

\noindent Theorem \ref{th:fi-pros-srs} shows that  the  FI matrix of the complete \BL{  PROS($n, S)$} sample  can be decomposed  into the FI matrix of the SRS data and a \BL{  non-negative definite} matrix, hence 
$\I_{pros}(\dt) \ge \I_{srs}(\dt)$. In other words,  complete PROS sample provides more information about the 
unknown parameters $\dt$ than SRS of the same size.
It is worth noting that the result of \Bc{\cite{chen2000efficiency}} and \Bb{\cite{Bara-Elsh-2001}} about FI of RSS data can be obtained  a  special case of Theorem \ref{th:fi-pros-srs} by setting $S=n$. 
We now compare the FI content of the complete  PROS sample
 with that of RSS of the same size about the unknown parameters $\dt$.
%  illustrates the fisher information 
% contained in PROS design is bigger than information that contained in RSS design.
\begin{theorem}\label{le:fi-pros-rss}
Under the conditions of Theorem \ref{th:fi-pros-srs},  the FI matrix of a  
complete \BL{ { \rm PROS($n, S$)}}  sample may be decomposed   as
\[
\I_{pros}(\dt)=\I_{rss}(\dt)+ \mathbb{H}(\dt),
\] 
where $\I_{rss}(\dt)$ is the FI matrix of an   RSS   of size $n$ (\Bc{when the set size is $n$}), and 
\[
\mathbb{H}(\dt)=n(S-n) 
\E\left\{
\frac{[D_{\dt}F(X;\dt)][D_{\dt}F(X;\dt)]^{\top}}{F(X;\dt) {\bar F}(X;\dt)} 
\right\},
\] 
is a non-negative definite matrix.
 \end{theorem}
\begin{proof} %The proof is  straightforward  and it is left to the reader. 
\BL{  Using Theorem \ref{th:fi-pros-srs} for $S=n$, we have
\[\I_{rss}(\dt)=\I_{srs}(\dt)+ n(n-1)  
\E\left\{
\frac{[D_{\dt}F(X;\dt)][D_{\dt}F(X;\dt)]^{\top}}{F(X;\dt) {\bar F}(X;\dt)}
\right\},
\]
where  $\I_{srs}(\dt)$ denotes 
the FI matrix of  a SRS of size $n$. Now, the result follows from the   above equation and the expression for  $\I_{pros}(\dt)$ in Theorem \ref{th:fi-pros-srs}.}
\end{proof}

\noindent Theorem \ref{le:fi-pros-rss} shows the superiority of a complete  PROS sample over an  RSS of the same size in terms of the FI content about the unknown vector of  parameters $\dt$. \Bb{ In comparing the Fisher information content of RSS data to that of SRS data, \cite{Bara-Elsh-2001} considered the example of point estimation within a location-scale family and the example of linear regression. We use the same two examples} to  obtain the FI content of \Ba {a complete PROS data set from the  location-scale family}  of distributions as well as a simple linear regression model and compare them  with those based on SRS and RSS data of the same size. To this end, let 
\[
RE_1(\dt)=\frac{det\{\I_{pros}(\dt)\}}{det\{\I_{srs}(\dt)\}}
\quad\text{and}\quad
RE_2(\dt)=\frac{det\{\I_{pros}(\dt)\}}{det\{\I_{rss}(\dt)\}}.
\]
From Theorems \ref{th:fi-pros-srs} and  \ref{le:fi-pros-rss} one can notice that   the set size $(S)$ and the number of the subsets $(n)$ 
are   two important parameters  of  \BL{  PROS($n, S)$} design that  influence the  FI  content of  PROS samples.  
 We  observe that increasing $S$ and $n$ 
 results in a considerable gain  in $RE_1$ and $RE_2$, respectively. Also, both $RE_1$ and $RE_2$ increase with the number of the parameters  of the  model.  \BL{   Later in this section we  investigate  the case where the set sizes are fixed in both PROS and RSS designs and consider the effect of  the number of subsets in PROS sampling design on the FI content of  PROS data compared with their RSS counterparts. }

\begin{example}\label{le:fi-lsf} (\textbf{Location-Scale family of distributions)}.
Under the assumptions of Theorem \ref{th:fi-pros-srs}, if $f(x; \dt)$ is a member of the location-scale family of distributions with pdf 
\[
f(x; \dt)=\frac{1}{\sigma}g(\frac{x-\mu}{\sigma}), \quad \dt=(\mu, \sigma)\in \mathbb{R}\times \mathbb{R}^{+}, 
\] 
where $g(\cdot)$ is a pdf with \Ba {corresponding} cdf $G(\cdot)$,   then
\begin{align*}\label{fi-srs-lsf}
\I_{pros}(\dt)=&  \I_{srs}(\dt) + \mathbb{K}(\dt) \\
 =& \frac{n}{\sigma^2}
\left(
       \begin{array}{cc}
       E\{\frac{{g^{'}(Z)}^2}{{g(Z)}^2}\} &  E\{\frac{Z{g^{'}(Z)}^2}{{g(Z)}^2}\} \\
       E\{\frac{Z{g^{'}(Z)}^2}{{g(Z)}^2}\} & E\{\frac{Z^2 {g^{'}(Z)}^2}{{g(Z)}^2}-1\}
       \end{array} 
\right) 
+ \frac{n(S-1)}{\sigma^2}
\left(
       \begin{array}{cc}
       E\{\frac{{g(Z)}^2}{G(Z)[1-G(Z)]}\} &  E\{\frac{Z{g(Z)}^2}{G(Z)[1-G(Z)]}\} \\
       E\{\frac{Z{g(Z)}^2}{G(Z)[1-G(Z)]}\} & E\{\frac{Z^2 {g(Z)}^2}{G(Z)[1-G(Z)]}\}
       \end{array} 
\right).
\end{align*}
If $f(x; \dt)$ is symmetric about the location parameter $\mu$, the FI matrix reduces to 
\begin{align*}%\label{fi-srs-sym}
\I_{pros}(\dt)%=&  \I_{srs}(\dt) + \I_{p.srs}(\dt) \\
 =& \frac{n}{\sigma^2}
\left(
       \begin{array}{cc}
       E\{\frac{{g^{'}(Z)}^2}{{g(Z)}^2}\} &  ~~~~~~~~~0 \\
       0 & E\{\frac{Z^2 {g^{'}(Z)}^2}{{g(Z)}^2}-1\}
   \end{array} 
\right) 
+ \frac{n(S-1)}{\sigma^2}
\left(
       \begin{array}{cc}
       E\{\frac{{g(Z)}^2}{G(Z)[1-G(Z)]}\} &  ~~~~~~~~~0 \\
       0 & E\{\frac{Z^2 {g(Z)}^2}{G(Z)[1-G(Z)]}\}
       \end{array} 
\right). 
\end{align*}
%As seen in Lemma \ref{le:fi-lsf} by setting $S=n$, we derive the finding \cite{Bara-Elsh-2001} 
%regarding FI content unde RSS for we only location-scale family of distributions.
%Due to the fact that \cite{Bara-Elsh-2001} obtained the results for ranked set sample on the basis of the simple random sample
%of the same size  for location-scale family of distributions, the Fisher information under PROS sample can easily be 
%compared with that and ranked set sample for the family of distributions.
%Note that the relative efficiencies of PROS sample do not 
%depend on the location and scale parameters of the distribution and only depends on the structure of the underlying distribution. 
%simply be expressed for all members of the family. To this end, 
\noindent \Bb{Similar to \cite{Bara-Elsh-2001} who compared the relative efficiency of RSS to SRS for some members of the location-scale family of distributions,} Tables \ref{ta:re-pros} shows
the values of $RE_1$ and $RE_2$  \Bb{ under the same distributions}.
%some members  of the location-scale family of distributions.   
As expected, the largest values of $RE_1$ and $RE_2$ are achieved in the cases where both location and scale parameters are considered to be unknown. 
\end{example}
%According to the fact that Barabesi ??? compared the RSS design with SRS design for location-scale family 
%therefore by using the results obtained in the Lemma??? and Lemma ??? obtained in the Barabesi ???, we can 
%easily express and compare PROS to RSS design for location-scale family.
%
%On the basis of Lemma ???, it is interesting to note that when the location-scale is considered, the relative 
%eficiency does not depend on the parameter.Hence in this case, the expression of the relative efficiency 
%turns out to be valid for the whole class.
%
%In order to give more insight into the performance of PROS with respect to RSS and SRS, the relative efficiency 
%was computed for some selected families of distributions and presented in Table ??? and Table ??? respectively. 
%From the analysis of the tables, it is at once apparent that large gains in efficiency with the use of PROS can 
%be obtained in most cases. The best results are achieved when the location and scale parameters are jointly 
%estimated- i.e. for the normal, logistic and extreme-value location-scale families. This feature is obviously 
%due to the fact that in these cases two parameters are considered, and hence- as has been previously remarked- 
%the relative efficiency is generally of order $n^2$. Lower gains in efficiency are obtained when the 
%one-parameter normal, logistic and extreme-value scale families are considered. However, for each family 
%considered, the gain in efficiency largely justifies the use of PROS when this protocol is suitable in practice.

{\begin{table}[h!]
\caption{{\footnotesize{The values of $RE_i(\dt)$,  $i=1, 2$ for comparing the FI content of the complete \BL{  PROS($n, S)$} sample with its   SRS and RSS  of the same size for some distributions. }}}
\vspace{0.3cm} % title name of the table
\centering % centering table
\footnotesize{\begin{tabular}{llllll}\hline\hline 
Distributions & Location    & Scale       & Shape       & $RE_1$     & $RE_2$                  \\\hline\hline 
Exponential   & 0           & $\sigma$    & -           & 1+0.4041$(S-1)$  &  1+0.4041$\{\frac{(S-n)}{1+0.4041(n-1)}\}$             \\[2ex]
Normal        & $\mu$       & 1           & -           & 1+0.4805$(S-1)$    & 1+0.4805$\{\frac{(S-n)}{1+0.4805(n-1)}\}$           \\
              & 0           & $\sigma$    & -           & 1+0.1350$(S-1)$     &  1+0.1350$\{\frac{(S-n)}{1+0.1350(n-1)}\}$          \\
              & $\mu$       & $\sigma$    & -           & 1+0.6155$(S-1)$ +0.0649$(S-1)^2$    &    1+$(\frac{0.6155(S-n)+0.0649[(S-1)^2-(n-1)^2]}
                                                             {1+0.6155(n-1)+0.0649(n-1)^2})$       \\[2ex]
Logistic      & $\mu$       & 1           & -           & 1+0.0050$(S-1)$      & 1+0.1666$\{\frac{(S-n)}{0.3332+0.1666(n-1)}\}$          \\
              & 0           & $\sigma$    & -           & 1+0.1513$(S-1)$     & 1+0.2149$\{\frac{(S-n)}{1.4189+0.2149(n-1)}\}$           \\
              & $\mu$       & $\sigma$    & -           & 1+0.6516$(S-1)$
                                                           +0.0757$(S-1)^2$       & 1+$(\frac{0.3081(S-n)+0.0358[(S-1)^2-(n-1)^2]}
                                                             {0.4728+0.3081(n-1)+0.0358(n-1)^2})$                 \\[2ex]
Extreme-value & $\mu$       & 1           & -           & 1+0.4041$(S-1)$       & 1+0.4041$\{\frac{(S-n)}{1+0.4041(n-1)}\}$         \\
              & 0           & $\sigma$    & -           & 1+0.2519$(S-1)$      & 1+0.2518$\{\frac{(S-n)}{1+0.2518(n-1)}\}$             \\
              & $\mu$       & $\sigma$    & -           & 1+0.6012$(S-1)$
                                                           +0.0686$(S-1)^2$    & 1+$(\frac{0.6560(S-n)+0.1017[(S-1)^2-(n-1)^2]}
                                                             {1+0.6560(n-1)+0.1017(n-1)^2})$            \\[2ex]
Gamma         & 0           & $\sigma$    & 2           & 1+0.4393$(S-1)$    & 1+0.7296$\{\frac{(S-n)}{1.6609+0.7296(n-1)}\}$           \\  
              & 0           & $\sigma$    & 3           & 1+0.4523$(S-1)$     & 1+1.1690$\{\frac{(S-n)}{2.5846+1.1690(n-1)}\}$                        \\  
              & 0           & $\sigma$    & 4           & 1+0.4591$(S-1)$    & 1+1.6161$\{\frac{(S-n)}{3.5200+1.6161(n-1)}\}$              \\  
              & 0           & $\sigma$    & 10           & 1+0.4718$(S-1)$    & 1+4.2396$\{\frac{(S-n)}{8.9820+4.2396(n-1)}\}$          \\\hline    \hline                                                    
\end{tabular}}
\label{ta:re-pros}
\end{table}
}

\begin{example} (\textbf{Linear Regression Model}).
In this example, \BL{ {\rm PROS($n, S$)}} sampling design is applied to the simple regression model $Y_i= \beta_0+ \beta_1 x_i+ \epsilon_i$  \Bc{with replicated observations of the response variable}
where for each  \Bc{value } $x_i$ of \Bc{independent variable}, $i=1,\ldots, k$, we have a 
PROS sample  of $Y$'s denoted by $\Bc{(Y_{i(d_1)},\ldots,Y_{i(d_n)})}$.
\Bc{For more details about the use of RSS sampling in this regression model, see \cite{barreto1999best} and \cite{Bara-Elsh-2001}}.
 Suppose $\epsilon_i$ are independent and identically distributed  random variables from a symmetric distribution with pdf $f(\cdot)$ and  cdf $F(\cdot)$, respectively.  Let   $E(\epsilon_i)=0$ and $Var(\epsilon_i)=\sigma^2$.
% Let the PROS samples be from a symmetric distribution with pdf 
%$\frac{1}{\sigma}f(\frac{y-a-bx_i}{\sigma})$ and cdf $F(\frac{y-a-bx_i}{\sigma})$.
 Without loss of generality, 
we take  ${\bar x}=\frac{1}{k}\sum_{i=1}^{k}x_i=0$, $s_x^2=\frac{1}{k}\sum_{i=1}^{k}x_i^2$ and let $\dt=(\beta_0, \beta_1, \sigma)$.
Using Example \ref{le:fi-lsf}, it is easy to show that 
%Proposing the use of 
%PROS in the simple linear regression model with replicated observations.  
%According to the approach, a PROS sample, say $(Y_{i[d_1]},Y_{i[d_2]},\ldots,Y_{i[d_n]})$, 
%is collected at each value $x_i$ of independent variable $(i=1,2,...,k)$. In turn, each of the $n$
%PROS samples is assumed to be relative to a symmetric random variable with distribution function
%$F\{(y-a-bx_i)/\sigma\}$ and probability density function $f\{(y-a-bx_i)/\sigma\}/\sigma$, $(i=1,2,...,k)$.
%Accordingly, by presuming that , if the assumption of theorem 1 hol;d, it can be easily shown that
%
\begin{eqnarray*}
\I_{srs}(\dt)&=&\sum_{i=1}^{k} \frac{n}{\sigma^2} 
\left(
\begin{array}{ccc}
E\{\frac{f^{'} (Z)^2}{f(Z)^2}\}      & x_i E\{\frac{f^{'}(Z)^2}{f(Z)^2}\} & 0\\
x_i E\{\frac{f^{'}(Z)^2}{f(Z)^2}\}  & x_i^2 E\{\frac{f^{'}(Z)^2}{f(Z)^2}\} & 0\\
0 & 0 & E\{\frac{Z^2 f^{'}(Z)^2}{f(Z)^2}\}-1
\end{array}
\right)\\
&=& \frac{nk}{\sigma^2}diag
\left(
E\{\frac{f^{'}(Z)^2}{f(Z)^2}\},s_x^2 E\{\frac{f^{'}(Z)^2}{f(Z)^2}\},E\{\frac{Z^2 f^{'}(Z)^2}{f(Z)^2}\}-1
\right),
\end{eqnarray*}
and
\begin{eqnarray*}
\mathbb{K}(\dt)&=& \sum_{i=1}^{k} \frac{2n(S-1)}{\sigma^2} 
\left(
\begin{array}{ccc}
E\{\frac{f(Z)^2}{F(Z)}\}      & x_i E\{\frac{f(Z)^2}{F(Z)}\} & 0\\
x_i E\{\frac{f(Z)^2}{F(Z)}\}  & x_i^2 E\{\frac{f(Z)^2}{F(Z)}\} & 0\\
0 & 0 & E\{\frac{Z^2 f^{'}(Z)^2}{F(Z)}\}
\end{array}
\right) \\
&=& \frac{2kn(S-1)}{\sigma^2} diag
\left(
E\{\frac{f(Z)^2}{F(Z)}\}, s_x^2 E\{\frac{f(Z)^2}{F(Z)}\}, E\{\frac{Z^2 f^{'}(Z)^2}{F(Z)}\}
\right).
\end{eqnarray*}
Note that $RE_1(\dt)$ is  independent of 
  $x_i$ and  $\dt$ and it only depends on the pdf $f(\cdot)$ and \Bc{the corresponding cdf} $F(\cdot)$.
 As \Ba{a special case}, when  $\epsilon_i$s are normally distributed,  one can easily show that   \[
RE_1(\dt)=\left\{1+0.4805(S-1)\right\}^2\left\{1+0.1350(S-1)\right\}.
\]
When $S=n$,  we obtain the result of   
 \cite{Bara-Elsh-2001} for  RSS data as \Ba{a special case} of our results.  
\end{example}

%%%%%%%%%%%%%%%%%%%%%%%%%
%%%%%%%%%%%%%%%%%%%%%%%%%
%%%%%%%%%%%%%%%%%%%%%%%%
\subsection{FI matrix of ${\bf X}_{pros}$ and the effect of misplacement errors  }\label{sub:error}
In this section we obtain the FI matrix of ${\bf X}_{pros}$.
\BL{  We study a setting when it is  assumed  that the subsetting process of PROS($n, S$) design   could be subjected to misplacement  errors between the subset groups.}
   For example, when the actual rank of a unit is in the judgment subset
$d_r$, due to judgment 
ranking error it could be misplaced into another judgment subset, say
$d_s$, $r\neq s$,  which leads to a  different kind of ranking error than the one usually \Ba{encountered}  in ranked set sampling. 
\BL{  Note that the FI matrix of  ${\bf X}_{pros}$ under perfect  subsetting  assumption can also be obtained as a special case of the   imperfect subsetting scenario.} 
We use the missing data model proposed  by \cite{Arsl} to model possible   misplacement errors in PROS sampling design. Let ${\bf X}_{pros}=\{ X_{[d_r]}, r=1, \ldots, n\}$ denote an imperfect PROS sample where $[\cdot]$ is used to show the presence  of misplacement errors in PROS  subsetting process. When the subsetting process is perfect we simply use $X_{(d_r)}$  to show PROS observations. 
Let  $\balpha$ denote the misplacement probability matrix,
\[
{\balpha}=
\left[
       \begin{array}{cccc}
       \alpha_{d_1,d_1} & \alpha_{d_1,d_2} & \ldots & \alpha_{d_1,d_n} \\
       \alpha_{d_2,d_1} & \alpha_{d_2,d_2} & \ldots & \alpha_{d_2,d_n} \\
       \vdots           & \vdots           & \Bc{\ddots} & \vdots           \\
       \alpha_{d_n,d_1} & \alpha_{d_n,d_2} & \ldots & \alpha_{d_n,d_n} \\
       \end{array}
\right]_{\BL{ {n\times n}}},
%\quad \text{and}\quad
%{\balpha}_{d_r,d_h}= \alpha_{d_r,d_h}
%\left[
%       \begin{array}{cccc}
%       \frac{1}{m^2} & \frac{1}{m^2}  & \ldots & \frac{1}{m^2} \\
%       \frac{1}{m^2} & \frac{1}{m^2}  & \ldots & \frac{1}{m^2} \\
%       \vdots      & \vdots       & \ldots & \vdots      \\
%       \frac{1}{m^2} & \frac{1}{m^2}  & \ldots & \frac{1}{m^2} \\
%       \end{array}
%\right]_{m\times m},
\]
where $\alpha_{d_r,d_h}$ is the misplacement probability of a unit from subset $d_h$ into subset $d_r$.
Since the design parameter $D$ creates a partition over the  sets, the matrix $\balpha$
should be a double stochastic matrix such that  $\sum_{r=1}^{n} \alpha_{d_r,d_h}=\sum_{h=1}^{n} \alpha_{d_r,d_h}=1$.
%Thus, the matrix $\balpha$ contains the misplacement probabilities that a unit at random from subset $d_h$ is 
%judged to be a unit at random from subset $d_r$. For more details, reader are refereed to \cite{Arsl} and \cite{Hate-Joza-Oztu-13}.
% a partition, we must have the constraint
% $\sum_{h=1}^{n} \alpha_{d_r,d_h}=1$. In order to have a valid probability model, we also
%need to have the constraint $\sum_{r=1}^{n} \alpha_{d_r,d_h}=1$. These two constraints define
%a stochastic matrix $\balpha$. 
%Thus, the matrix $\balpha$ contains the misplacement probabilities
%of an order statistic from  the subset $d_h$ to a judgment order statistic in another subset $d_r$,
%for $h,r\in\{1,\ldots,n\}$. Entries $1/m^2$ indicate that random selection from $d_h$ and random replacement to $d_r$
%are all  equally likely.
Suppose  $f_{[d_r]}(\cdot;\dt)$ is  the pdf  of $X_{[d_r]}$, $r=1, \ldots, n$. One can easily show that  
\begin{eqnarray}\label{f-imp} 
f_{[d_r]}(x_{[d_r]};\dt) = 
\sum_{h=1}^{n}
\alpha_{d_r,d_h} f_{(d_h)}(x_{[d_r]};\dt) 
%\frac{1}{m} \sum_{h=1}^{n} \sum_{u\in d_r}^{}
%\alpha_{d_r,d_h} f^{(u:S)}(x_{[d_r]};\dt) \\
=  
f(x_{[d_r]};\dt) g_r(x_{[d_r]};\dt),
\end{eqnarray}
where 
\begin{align}\label{eq:gr}
g_r(x;\dt)= n \sum_{h=1}^{n} \sum_{u\in d_h}^{} \alpha_{d_r,d_h}
{{S-1}\choose{u-1}} [F(x;\dt)]^{u-1} [1-F(x;\dt)]^{S-u}.
\end{align}
%
%It is worth noting that the previous model assumes that the statistic obtained from subset $d_r$ and the 
%event that it receives judgment subset $h$ are independent. Even if this assumption barely holds in 
%practice, since the group units which have similar and close values of the variable are more difficult 
%to partition them correctly than those whose values are clearly separated, the model is useful for 
%analyzing the theoretical effects of imperfect PROS.
The likelihood function under an  imperfect \BL{  PROS($n, S)$} design     is now  given by
\begin{eqnarray*}\label{l-imp}
L(\Omega)
= \prod_{r=1}^{n} f_{[d_r]}(x_{[d_r]};\dt)
%&=& \prod_{r=1}^{n} 
%\left\{ 
%\frac{1}{m} \sum_{h=1}^{n} \sum_{u\in d_r}^{}
%\alpha_{d_r,d_h} f^{(u:S)}(x_{[d_r]};\Psi)
%\right\}  \proleft\{ 
= \prod_{r=1}^{n} f(x_{[d_r]};\dt) g_r(x_{[d_r]};\dt),
\end{eqnarray*}
where $\Omega= (\dt, \balpha)$. 
%where 
%\[
%g(x;\Psi)= \sum_{h=1}^{n} \sum_{u\in d_h}^{} \alpha_{d_r,d_h} \frac{S}{m}
%{{S-1}\choose{u-1}} [F(x;\Psi)]^{u-1} [1-F(x;\Psi)]^{S-u}
%\]
To obtain the FI matrix of an imperfect  PROS sample and compare it with its SRS and RSS counterparts
we need the following result,  the proof of which is left to the reader.
\begin{lemma} \label{le:f-imp}
Let $Y_r=X_{[d_r]}$,   $r=1,\ldots, n$, be observed from a continuous  distribution with pdf  $f(\cdot; \dt)$ using  an imperfect \BL{ {\rm PROS($n, S$)}} sampling design.  Suppose  $f_{[d_r]}(\cdot;\dt)$ and $g_r(\cdot,\dt)$ are  defined 
as in \eqref{f-imp} and \eqref{eq:gr}, respectively. Under the regularity conditions \BL{  of {\rm \cite{chen-bai-sinha}}}, we have

\begin{itemize}

\item [(i)] $\sum_{r=1}^{n} f_{[d_r]}(x;\dt)= nf(x;\dt)$,
\item [(ii)] $\sum_{r=1}^{n} g_r(x;\dt)=n$,
\item [(iii)] 
 $\sum_{r=1}^{n}E
\left\{
\frac{D^2_{\dt}g_r(Y_r;\dt)}{g_r(Y_r;\dt)}
\right\}=0,
$
\item [(iv)]
$
\sum_{r=1}^{n} E
\left\{
\frac{[D_{\dt} g_r(Y_r;\dt)][D_{\dt} g_r(Y_r;\dt)]^{\top}}{g_r^2(Y_r;\dt)}
\right\}
= \BL{  \sum_{r=1}^{n}}
E
\left\{
\frac{[D_{\dt} g_r(X;\dt)][D_{\dt} g_r(X;\dt)]^{\top}}{g_r(X;\dt)}
\right\}.
$
\end {itemize}
\end{lemma}
%\noindent{\bf Proof}. The proof is presented in the Appendix.
\noindent Now,  we  show that the FI content of   ${\bf X}_{pros}$ is more that its SRS counterpart.  Unfortunately, it is hard to obtain analytical results  to compare the FI content of PROS and RSS data, therefore,  we   should rely on \Bc{numerical studies} for this case (see Tables \ref{ta:re-imp-6} and \ref{ta:re-imp-12}). 
\begin{theorem}\label{th:fi-imp}
Under the conditions of Lemma \ref{le:f-imp}, the FI matrix of an  imperfect \BL{ {\rm PROS($n, S$)}} sample about 
unknown parameters $\Omega=(\balpha,\dt)$  is given by  
\begin{eqnarray*}\label{th:fis-imp}
\I_{ipros}(\Omega)
&=&\I_{srs}(\dt)+\sum_{r=1}^{n} 
E
\left\{
\frac{[D_{\dt} g_r(X;\dt)][D_{\dt} g_r(X;\dt)]^{\top}}{g_r(X;\dt)}
\right\}\\
&=& \I_{srs}(\dt)+\sum_{r=1}^{n} \tilde\Delta_r,
\end{eqnarray*}
where $\sum_{r=1}^{n} \tilde\Delta_r$ is a \BL{  non-negative} definite matrix. 
%Therefore, the FI of 
%incomplete data under imperfect PROS sampling is larger than SRS data of the same size.
%\noindent{\bf Proof.} From Lemma \ref{le:f-imp}, the proof is straightforward and is omitted.
\end{theorem}
\begin{proof} The proof is similar to the proof of Theorem \ref{th:fi-pros-srs} and hence it is omitted. 
\end{proof}
 \noindent  To  study the effect of misplacement errors  in the subsetting process of  \BL{  PROS($n, S)$}  design   on the information content of the sample, \Bb{following 
 \cite{Bara-Elsh-2001},} we   consider 
 the following misplacement probabilities matrices  when  $n=2$ and $n=3$, 
\[
{\balpha}_1=
\left[
       \begin{array}{cc}
       p & 1-p \\
       1-p & p \\
       \end{array}
\right]
\quad \text{and}\quad
{\balpha}_2=
\left[
       \begin{array}{ccc}
       p & \frac{1-p}{2} & \frac{1-p}{2} \\
       \frac{1-p}{2} & p & \frac{1-p}{2} \\
       \frac{1-p}{2} & \frac{1-p}{2} & p \\
       \end{array}
\right].
\]
For some members of the location-scale family of distributions,   numerical values of $RE_1(\dt)$ and $RE_2(\dt)$ are calculated to compare the FI content of imperfect    PROS samples  with their    SRS and RSS counterparts of the same size when  $S=6$ and $S=12$. These values  are reported in Tables \ref{ta:re-imp-6} and \ref{ta:re-imp-12},  respectively. 
\Bc{ The results are calculated through a Monte Carlo simulation study comprising of 50,000 replications.} 
Both tables show that  misplacement   errors in  the subsetting process of   PROS 
sampling  have considerable effect on the information content of   PROS data about the unknown parameters of the model.
Note that,  when the subsetting  process is done randomly, i.e.,  
 $p=1/2$ when $n=2$ and $p=1/3$ in the case $n=3$,  the FI content of PROS samples  is the same as the FI content of SRS and RSS data of the same size. 
\Bb{ Similar results for comparing the FI content of imperfect  RSS   and    SRS samples can be found in \cite{Bara-Elsh-2001}.}

\BL{ 
Now, we   investigate  the effect of PROS sampling parameters $S$ and $n$ on the FI content of PROS samples  compared with their RSS counterparts.  To this end, we first calculate the FI content of two   ranked set samples  with fixed set sizes  6 and 12  when the cycle size is 1, under both perfect and different imperfect ranking scenarios. The FI content of  RSS samples are then compared with that of  PROS samples under different values of $S, n$ and $N$, where $N$ is the number of cycles   in order to match the  number of PROS  observations with their corresponding RSS samples.  Under some members of the location-scale family of distributions, Tables \ref{ta:imp-pros-rss-6} and \ref{ta:imp-pros-rss-12} provide the values of $RE_2(\dt)$  for the sample sizes 6 and 12, respectively, where  the  subsetting and ranking error \Ba{probability} matrices are defined   following the same  structure used in   $\balpha_1$ and $\balpha_2$ with proper adjustments  to the off-diagonal elements for  the set size. 
\Bc{For example,  consider the case where $S=6, n=3, l=2$ in Table \ref{ta:imp-pros-rss-6}. In this case, RSS design with set size $S=6$ is  compared with  the PROS design with set size $S=6$,  each consisting of three subsets  $n=3$ of equal sizes $m=2$. Since the PROS design results in 3 observations (as opposed to RSS  that results in  6 observations), PROS sampling is replicated with two cycles $l=2$. The relative efficiency values are simulated through  a Monte Carlo study with 50,000 replications.}
From Tables \ref{ta:imp-pros-rss-6} and \ref{ta:imp-pros-rss-12}, it is at once apparent that sampling parameters $S$ and $n$ as well as ranking (subsetting) error  models play key roles on the information  content of PROS data about unknown parameters of the model. As noted earlier, one observes that the performance of PROS$(n,S)$ and RSS coincides when $S=n$. We also note that for  fixed set size  $S$  (in both RSS and PROS design)  and under moderately accurate ranking in RSS design,  some PROS samples carry less information than  RSS of the same size about the  parameter of  the underlying population. However, the difference between the information content of PROS and RSS data \Bc{diminishes as $n$ increases to $S$}.  One may also observe  more informative PROS samples than RSS data of the same size (even with a larger set size than that of PROS design)  when the ranking error in RSS design is large. 
 %
%\begin{remark}  Suppose ${\bf X}_{rss}$ is a ranked set sample of size $S$ obtained from an RSS design with set size $S$. Let ${\bf X}_{pros}$ denote a PROS sample of the same size   from the underlying population when  in each set is $S$,  the number of subsets  is $n$ and $N$ is the number of cycles such that $nN=S$.  For a scalar parameter $\theta$,  using  Theorems \ref{th:fi-pros-srs} and  \ref{le:fi-pros-rss}, one can easily show that 
%\begin{align*}
%RE_2(\theta) = 1 + (S-1)\, (n-S)\frac{ \mathbb{K}(\theta)}{  \mathbb{I}_{rss}(\theta)}=\left\{\begin{array}{cc}  <1 & \text{ if }  n<S,  \\
%=1 & \text{ if } n=S.
%\end{array}\right. 
%\end{align*}
%In other words,  for a fixed set size in both PROS and RSS designs, RSS  data carry more information about the unknown parameter of interest than PROS data.  Note that this result is  valid under perfect ranking and subsetting in RSS and PROS designs. For imperfect ranking situations, as we observe in Tables \ref{ta:imp-pros-rss-6} and \ref{ta:imp-pros-rss-12},  one might get more informative PROS data than RSS data of the same size even with equal or smaller set sizes. This shows that PROS design is less sensitive to the subsetting errors that RSS is with respect to imperfect ranking. 
%\end{remark}
}

%%%%%%%%%%%%%%%%%%%%%%%%%%%%%%%
\subsection{FI using the Dell and Clutter model for misplacement ranking errors}
\Bc{Here, we propose  two-stage Monte Carlo simulations to study the effect of misplacement ranking error models on the FI content of PROS samples. Following the model proposed in  \cite{dell1972ranked}, in the first stage we compute the misplacement probabilities of PROS and RSS designs. In the second stage, these misplacement probabilities are used to compute the  FI content of PROS and RSS sampling designs.
}
\Bc{ Using  the Dell and Clutter  model for $\rho=1,0.9,0.75,0.5,0.25$ (representing different degrees of association between the ranking covariate and the response variable), the first stage  computes the misplacement probabilities matrices (${\bf\alpha}_i=1,\ldots,5$) for each $\rho$ through simulations of size 5000. 
  Using the estimated misplacement probabilities, in the second stage, we compute the FI content of the PROS, RSS and SRS sampling designs through Monte Carlo simulations comprising of 50,000 replicates. The results of the simulation studies for different family of distributions ( like previous simulation studies) are reported in Tables \ref{ta:re-imp-b-dc} and \ref{ta:imp-dlfixtab}. To explore the effect ranking errors on different distributions, we also computed the FI content of PROS samples under four different  mixture of two univariate exponential distributions
$
f(x;{\bf\Psi})=\pi \alpha e^{-\alpha x}+(1-\pi) \beta e^{-\beta x}, \quad x>0, 
$
where $\pi \in(0, 1)$,  $\alpha, \beta>0$ and ${\bf \Psi}= (\pi, \alpha, \beta)$.
To handle the mixture of exponential distributions, following \cite{hill1963information}, we calculated the numerical values of the relative efficiencies.
To do so, a new parameter  $h=\frac{\alpha}{\beta}$ is introduced and the exponential mixture model with three parameters $(\pi, \alpha, \beta)$ is transformed to a  mixture density  with two parameters $(\pi, h)$. }

%%%%%%%%%%%%%%%%%%%%%%%%%%%%%%%

\small{\begin{table}[h!]
\caption{\Bc{{Values of $RE_1$ and $RE_2$ to compare the FI content of imperfect PROS data with its SRS and RSS counterparts  of the same size for some  distributions when $S=6$.}}}
\vspace{0.3cm} % title name of the table
\centering % centering table
{
\Bc{\begin{tabular}{lllccccccccccc}\hline
&         &&\multicolumn{11}{c}{$p$}      \\
\cline{4-14}
Distribution  & $n$   &Design &0    &0.1  &0.2  &0.3  &0.4  &0.5  &0.6  &0.7  &0.8  &0.9  &1      \\\hline
Normal        &2      &$RE_1$ &2.48 &1.67 &1.34 &1.14 &1.03 &1.000 &1.03 &1.14 &1.34 &1.67 &2.48   \\[-1.7ex]
              &       &$RE_2$ &1.47 &1.25 &1.14 &1.06 &1.02 &1.000 &1.02 &1.06 &1.14 &1.25 &1.47     \\
              &3      &$RE_1$ &1.82 &1.28 &1.08 &1.004 &1.02 &1.11 &1.28 &1.54 &1.94 &2.54 &3.78    \\[-1.7ex]
              &       &$RE_2$ &1.28 &1.11 &1.03 &1.002 &1.01 &1.04 &1.11 &1.19 &1.28 &1.38 &1.54       \\\hline
Exponential   &2      &$RE_1$ &1.93 &1.47 &1.24 &1.10 &1.02 &1.000 &1.02 &1.10 &1.24 &1.47 &1.93    \\[-1.7ex]
              &       &$RE_2$ &1.37 &1.20 &1.11 &1.05 &1.01 &1.000 &1.01 &1.05 &1.11 &1.20 &1.37     \\
              &3      &$RE_1$ &1.47 &1.18 &1.05 &1.003 &1.01 &1.07 &1.18 &1.35 &1.58 &1.90 &2.44     \\[-1.7ex]
              &       &$RE_2$ &1.18 &1.08 &1.02 &1.001 &1.01 &1.03 &1.07 &1.13 &1.19 &1.26 &1.36     \\\hline
Logistic      &2      &$RE_1$ &2.73 &1.78 &1.39 &1.16 &1.04 &1.000 &1.04 &1.16 &1.39 &1.78 &2.73    \\[-1.7ex]
              &       &$RE_2$ &1.58 &1.30 &1.17 &1.08 &1.02 &1.000 &1.02 &1.08 &1.17 &1.30 &1.58     \\
              &3      &$RE_1$ &1.88 &1.31 &1.09 &1.005 &1.02 &1.12 &1.31 &1.61 &2.06 &2.74 &4.14     \\[-1.7ex]
              &       &$RE_2$ &1.32 &1.12 &1.04 &1.002 &1.01 &1.05 &1.12 &1.21 &1.31 &1.43 &1.61      \\\hline                                                                
%
%Exponential   &2   &$RE_1$ &&&&&&&&&&&    \\[-1.7ex]
%              &    &$RE_2$ &&&&&&&&&&&     \\
%              &3   &$RE_1$ &&&&&&&&&&&     \\[-1.7ex]
%              &    &$RE_2$ &&&&&&&&&&&      \\\hline
\end{tabular}
\label{ta:re-imp-6}}}
\end{table}}

{\begin{table}[h!]
\caption{\Bc{{Values of $RE_1$ and $RE_2$ to compare the FI content of imperfect PROS data with its SRS and RSS counterparts  of the same size for some distributions when $S=12$. }}}
\vspace{0.3cm} % title name of the table
\centering % centering table
{\Bc{\begin{tabular}{lllccccccccccc}\hline
&         &&\multicolumn{11}{c}{$p$}      \\
\cline{4-14}
Distribution  &$n$   &Design &0    &0.1  &0.2  &0.3  &0.4  &0.5    &0.6  &0.7  &0.8  &0.9  &1      \\\hline
Normal        &2     &$RE_1$ &3.15 &1.96 &1.48 &1.20 &1.05 &1.000 &1.05 &1.20 &1.48 &1.96 &3.15    \\[-1.7ex]
              &      &$RE_2$ &1.87 &1.46 &1.26 &1.12 &1.03 &1.000 &1.03 &1.12 &1.26 &1.46 &1.87     \\
              &3     &$RE_1$ &2.51 &1.49 &1.13 &1.007 &1.03 &1.18 &1.46 &1.90 &2.56 &3.58 &5.74     \\[-1.7ex]
              &      &$RE_2$ &1.77 &1.29 &1.09 &1.005 &1.02 &1.11 &1.26 &1.46 &1.68 &1.93 &2.32     \\\hline
Exponential   &2     &$RE_1$ &2.39 &1.69 &1.35 &1.15 &1.04 &1.000 &1.04 &1.15 &1.35 &1.69 &2.39    \\[-1.7ex]
              &      &$RE_2$ &1.70 &1.38 &1.21 &1.10 &1.02 &1.000 &1.02 &1.10 &1.21 &1.38 &1.70    \\
              &3     &$RE_1$ &1.85 &1.31 &1.09 &1.005 &1.02 &1.12 &1.30 &1.57 &1.93 &2.43 &3.30     \\[-1.7ex]
              &      &$RE_2$ &1.48 &1.19 &1.06 &1.003 &1.01 &1.08 &1.18 &1.31 &1.45 &1.61 &1.82     \\\hline
Logistic      &2     &$RE_1$ &3.56 &2.14 &1.57 &1.24 &1.06 &1.000 &1.06 &1.24 &1.57 &2.14 &3.56     \\[-1.7ex]
              &      &$RE_2$ &2.06 &1.56 &1.32 &1.15 &1.04 &1.000 &1.04 &1.15 &1.32 &1.56 &2.06      \\    
              &3     &$RE_1$ &2.72 &1.55 &1.15 &1.008 &1.03 &1.20 &1.53 &2.04 &2.81 &4.04 &6.65    \\[-1.7ex]
              &      &$RE_2$ &1.90 &1.33 &1.10 &1.005 &1.02 &1.13 &1.30 &1.53 &1.79 &2.09 &2.56     \\\hline
%
%
%
%Logistic      &2   &$RE_1$ &&2.73  &1.55 &1.15 &1    &1.03 &1.20 &1.53 &2.04 &2.83 &4.07 &6.72    \\[-1.7ex]
%              &    &$RE_2$ &&1.90  &1.33 &1.10 &1    &1.02 &1.12 &1.30 &1.53 &1.80 &2.10 &2.57     \\
%              &3   &$RE_1$ &&3.56  &2.13 &1.56 &1.23 &1.05 &1    &1.05 &1.23 &1.56 &2.13 &3.56     \\[-1.7ex]
%              &    &$RE_2$ &&2.06  &1.56 &1.31 &1.14 &1.03 &1    &1.03 &1.14 &1.31 &1.56 &2.06      \\\hline              
\end{tabular}
\label{ta:re-imp-12}}}
\end{table}
}

\Bc{In the next section, we  study  the uncertainty structure (as another aspect of information content) of PROS
samples in terms of some well-known measures including Shannon entropy, R\'enyi entropy and KL information. Nevertheless, it is worth mentioning that the  FI and uncertainty content play important roles in
different inferential aspects of the PROS sampling designs including, for instance,
maximum likelihood (ML) estimation  and its properties.
The FI matrix is a key concept in the theory of statistical inference particularly in the theory of ML estimation problem \citep{lehmann1998theory}. It is used to derive asymptotic distribution of MLE and to calculate the covariance matrices associated with ML estimates as well as  Bayesian Statistics.}
 
%%%%%%%%%%%%%%%%%%%%%%%%%%%%%%%%%%%%%%%%%%%%%%%%%%%%%%%%%% % New Tables 

% latex table generated in R 3.0.2 by xtable 1.7-1 package
% Fri Jan  9 20:23:45 2015
{\begin{table}[h!]
\caption{\Bc{\footnotesize{Values of $RE_1$ and $RE_2$ to compare the FI content of imperfect PROS data with its SRS and RSS counterparts  of the same size for some distributions based on different Dell-Clutter parameters when $S\in\{6,12\}$. }}}
\vspace{0.3cm} % title name of the table
\centering % centering table
\footnotesize{\Bc{\begin{tabular}{llcccccccccccc}\hline
%              &   &       &\multicolumn{5}{c}{$S=6$}   &  &\multicolumn{5}{c}{$S=12$}   \\
%              \cline{4-8} \cline{10-14}
              &   &       &\multicolumn{5}{c}{S=6,~$\rho$}   &  &\multicolumn{5}{c}{S=12,~$\rho$}   \\
\cline{4-8} \cline{10-14}              
Distribution  &$n$&Design &$0.25$&$0.50$&$0.75$&$0.90$&$1.00$& &$0.25$&$0.50$&$0.75$&$0.90$&$1.00$ \\\hline
Normal        & 2 & $RE_1$& 1.02 & 1.10 & 1.27 & 1.51 & 2.48 & & 1.03 & 1.13 & 1.36 & 1.68 & 3.15 \\[-1.7ex] 
              &   & $RE_2$& 1.00 & 1.02 & 0.98 & 0.96 & 1.47 & & 1.01 & 1.04 & 1.05 & 1.06 & 1.87 \\ 
              & 3 &$RE_1$ & 1.03 & 1.15 & 1.43 & 1.85 & 3.70 & & 1.04 & 1.20 & 1.57 & 2.16 & 5.75 \\[-1.7ex] 
              &   &$RE_2$ & 1.01 & 1.02 & 1.02 & 1.05 & 1.50 & & 1.01 & 1.07 & 1.12 & 1.23 & 2.32 \\\hline 
Exponential   & 2 & $RE_1$& 1.02 & 1.06 & 1.18 & 1.31 & 1.92 & & 1.02 & 1.08 & 1.22 & 1.44 & 2.38 \\ [-1.7ex]
              &   & $RE_2$& 1.00 & 1.00 & 0.99 & 0.96 & 1.37 & & 1.01 & 1.02 & 1.03 & 1.06 & 1.69 \\ 
              & 3 &$RE_1$ & 1.03 & 1.11 & 1.30 & 1.56 & 2.47 & & 1.03 & 1.14 & 1.37 & 1.73 & 3.44  \\ [-1.7ex] 
              &   &$RE_2$ & 1.00 & 1.02 & 1.02 & 1.04 & 1.35 & & 1.01 & 1.05 & 1.07 & 1.16 & 1.89  \\\hline 
Logistic      & 2 & $RE_1$& 1.02 & 1.10 & 1.31 & 1.55 & 2.69 & & 1.04 & 1.16 & 1.40 & 1.78 & 3.54  \\  [-1.7ex]
              &   & $RE_2$& 1.00 & 1.00 & 1.01 & 0.96 & 1.56 & & 1.01 & 1.06 & 1.08 & 1.10 & 2.05  \\ 
              & 3 &$RE_1$ & 1.03 & 1.16 & 1.49 & 1.95 & 4.13 & & 1.04 &  1.21 & 1.64 & 2.28 & 6.76  \\  [-1.7ex]
              &   &$RE_2$ & 1.00 & 1.01 & 1.04 & 1.06 & 1.60 & & 1.01 & 1.06 & 1.15 & 1.24 & 2.62  \\ 
              \hline
$\Psi=(\pi,h)$& 2 & $RE_1$& 1.04 & 1.13 & 1.44 & 1.80 &3.83  & &1.05  & 1.21 & 1.51 & 2.02 &5.11  \\ [-1.7ex]
$(0.3,1/3)$   &   & $RE_2$& 1.01 & 0.98 & 0.95 & 0.90 &2.36  & &1.02  & 1.04 & 1.00 & 1.00 & 3.16 \\ 
              & 3 &$RE_1$ & 1.06 & 1.22 & 1.59 & 2.07 &4.51  & &1.06  & 1.23 & 1.65 & 2.30 & 6.44 \\ [-1.7ex]
              &   &$RE_2$ & 1.01 & 1.02 & 1.00 & 1.00 &1.91  & &1.02  & 1.03 & 1.04 & 1.11 & 2.73\\
              \hline 
$\Psi=(\pi,h)$& 2 & $RE_1$& 1.02 & 1.10 & 1.26 & 1.48 &3.39  & &1.03  & 1.10 & 1.25 & 1.50 & 5.25 \\[-1.7ex] 
$(0.3,1/9)$   &   & $RE_2$& 1.00 & 0.99 & 0.92 & 0.85 &2.10  & &1.00  & 0.98 & 0.91 & 0.86 & 3.25 \\ 
              & 3 &$RE_1$ & 1.03 & 1.14 & 1.41 & 1.72 &4.45  & &1.04  & 1.15 & 1.42 & 1.71 & 7.59 \\ [-1.7ex]
              &   &$RE_2$ & 1.00 & 0.99 & 0.97 & 0.92 &1.97  & &1.01  & 1.00 & 0.97 & 0.92 & 3.35 \\ 
              \hline
$\Psi=(\pi,h)$& 2 & $RE_1$& 1.05 & 1.19 & 1.50 & 2.02 &3.67    & &1.05  & 1.22 & 1.61 & 2.16 & 4.34 \\ [-1.7ex] 
$(0.9,1/3)$   &   & $RE_2$& 1.01 & 1.02 & 0.97 & 0.91 &2.15    & & 1.01 & 1.05 & 1.04 & 0.98 & 2.54 \\ 
              & 3 &$RE_1$ & 1.04 & 1.20 & 1.56 & 2.13 &5.24    & &1.06  & 1.25 & 1.70 & 2.39 & 6.60 \\ [-1.7ex] 
              &   &$RE_2$ & 1.00 & 1.01 & 0.97 & 1.01 &2.05    & & 1.02 & 1.05 & 1.06 & 1.13 &  2.59\\ 
              \hline 
$\Psi=(\pi,h)$& 2 & $RE_1$& 1.02 & 1.09 & 1.23 & 1.46 & 2.85 & &1.03 & 1.11 & 1.27 & 1.57 & 3.57 \\ [-1.7ex] 
$(0.9,1/9)$   &   & $RE_2$& 1.00 & 0.98 & 0.91 & 0.83 & 1.74 & &1.01 & 1.00 & 0.94 & 0.89 & 2.18 \\ 
              & 3 &$RE_1$ & 1.03 & 1.12 & 1.36 & 1.76 & 4.33 & &1.04 & 1.16 & 1.44 & 1.86 & 6.95 \\ [-1.7ex] 
              &   &$RE_2$ & 1.00 & 1.00 & 0.97 & 0.98 & 1.84 & &1.01 & 1.03 & 1.02 & 1.03 & 2.96 \\ 
   \hline                
\end{tabular}
\label{ta:re-imp-b-dc}}}
\end{table}
}

\begin{table}[h!]
\caption{{\footnotesize{\Bc{Values of $RE_2$ to compare the FI content of imperfect PROS$(n,S)$ with imperfect RSS of a fixed set size $S\in\{6,12\}$ under different Dell-Clutter Model}}}}
\vspace{0.3cm} % title name of the table
\centering % centering table
\footnotesize{\Bc{\begin{tabular}{llllcccccccccccccc}\hline
&             &    &   & \multicolumn{5}{c}{$S=6,~\rho$}     & &   &   &   &  \multicolumn{5}{c}{$S=12,~\rho$}    \\
\cline{5-9} \cline{14-18} 
Distribution  &$S$ &$n$&$N$& 0.25 & 0.50 &  0.75& 0.90 & 1.00  & &$S$ &$n$&$N$& 0.25 & 0.50 &  0.75& 0.90 & 1.00     \\\hline
Normal        & 4  & 2 & 3 & 0.97 & 0.88 & 0.73 & 0.58 & 0.39& & 6 & 2 & 6 & 0.97 & 0.84 & 0.61 & 0.40 & 0.16 \\ [-1.7ex] 
              & 6  & 2 & 3 & 0.98 & 0.89 & 0.75 & 0.60 & 0.44& & 6 & 3 & 4 & 0.98 & 0.90 & 0.70 & 0.49 & 0.25 \\ [-1.7ex]
              & 6  & 3 & 2 & 0.99 & 0.94 & 0.86 & 0.75 & 0.67& &12 & 2 & 6 & 0.97 & 0.87 & 0.67 & 0.45 & 0.21 \\ [-1.7ex]
              & 8  & 2 & 3 & 0.98 & 0.90 & 0.78 & 0.62 & 0.49& &12 & 3 & 4 & 0.98 & 0.91 & 0.77 & 0.57 & 0.39 \\ [-1.7ex]
              & 12 & 2 & 3 & 0.99 & 0.91 & 0.82 & 0.68 & 0.56& &12 & 4 & 3 & 0.99 & 0.95 & 0.83 & 0.68 & 0.55 \\ [-1.7ex]
              & 12 & 3 & 2 & 0.99 & 0.98 & 0.93 & 0.86 & 1.02& &12 & 6 & 2 & 1.00 & 0.99 & 0.92 & 0.85 & 0.81 \\ [-1.7ex]
              & 12 & 6 & 1 & 1.01 & 1.03 & 1.11 & 1.26 & 2.04& &12 &12 & 1 & 1.00 & 1.01 & 1.01 & 1.03 & 1.03 \\ 
%              & 18 & 2 & 3 & 0.99 & 0.96 & 0.88 & 0.75 & 0.65& &18 &2  & 6 & 0.98 & 0.91 & 0.70 & 0.48 & 0.24 \\
%              & 18 & 3 & 2 & 1.00 & 0.99 & 0.98 & 0.99 & 1.35& &18 & 3 & 4 & 0.99 & 0.92 & 0.81 & 0.64 & 0.50 \\ 
%              & 18 & 6 & 1 & 1.01 & 1.07 & 1.18 & 1.47 & 3.19& &18 & 6 & 2 & 1.00 & 0.99 & 0.97 & 0.96 & 1.20 \\     
  \hline
Exponential   & 4  & 2 & 3 & 1.00 & 1.00 & 1.02 & 1.02 & 1.03& & 6  & 2 & 6 & 0.96 & 0.84 & 0.65 & 0.52 & 0.35 \\[-1.7ex] 
              & 6  & 2 & 3 & 0.98 & 0.90 & 0.78 & 0.69 & 0.65& & 6  & 3 & 4 & 0.97 & 0.87 & 0.73 & 0.62 & 0.45 \\ [-1.7ex]
              & 6  & 3 & 2 & 0.99 & 0.95 & 0.88 & 0.82 & 0.83& & 12 & 2 & 6 & 0.96 & 0.86 & 0.68 & 0.57 & 0.44 \\ [-1.7ex]
              & 8  & 2 & 3 & 0.98 & 0.91 & 0.80 & 0.71 & 0.70& & 12 & 3 & 4 & 0.97 & 0.91 & 0.76 & 0.67 & 0.62 \\[-1.7ex] 
              & 12 & 2 & 3 & 0.98 & 0.92 & 0.82 & 0.75 & 0.80& & 12 & 4 & 3 & 0.98 & 0.94 & 0.84 & 0.76 & 0.76 \\[-1.7ex] 
              & 12 & 3 & 2 & 0.99 & 0.97 & 0.92 & 0.90 & 1.14& & 12 & 6 & 2 & 0.99 & 0.97 & 0.91 & 0.87 & 0.88 \\ [-1.7ex]
              & 12 & 6 & 1 & 1.01 & 1.04 & 1.09 & 1.16 & 1.61& & 12 & 12& 1 & 1.00 & 1.00 & 0.99 & 0.99 & 0.99 \\ 
%              & 18 & 2 & 3 & 0.99 & 0.93 & 0.85 & 0.80 & 0.92& & 18 & 2 & 6 & 0.96 & 0.87 & 0.71 & 0.60 & 0.51 \\ 
%              & 18 & 3 & 2 & 0.99 & 0.99 & 0.96 & 0.96& & 1.38 & 18 & 3 & 4 & 0.98 & 0.92 & 0.80 & 0.72 & 0.75 \\ 
%              & 18 & 6 & 1 & 1.02 & 1.07 & 1.14 & 1.26 & 2.14& & 18 & 6 & 2 & 1.00 & 1.00 & 0.96 & 0.96 & 1.17 \\ 
  \hline
Logistic      & 4  & 2 & 3 & 0.97 & 0.88 & 0.71 & 0.56 & 0.36& & 6  & 2 & 6 & 0.96 & 0.82 & 0.59 & 0.38 & 0.16 \\ [-1.7ex]
              & 6  & 2 & 3 & 0.97 & 0.89 & 0.72 & 0.58 & 0.43& & 6  & 3 & 4 & 0.98 & 0.88 & 0.69 & 0.47 & 0.24 \\ [-1.7ex]
              & 6  & 3 & 2 & 0.99 & 0.94 & 0.85 & 0.74 & 0.66& & 12 & 2 & 6 & 0.96 & 0.87 & 0.64 & 0.43 & 0.20 \\ [-1.7ex]
              & 8  & 2 & 3 & 0.97 & 0.92 & 0.76 & 0.61 & 0.49& & 12 & 3 & 4 & 0.99 & 0.90 & 0.74 & 0.55 & 0.38 \\ [-1.7ex]
              & 12 & 2 & 3 & 0.98 & 0.93 & 0.80 & 0.65 & 0.56& & 12 & 4 & 3 & 0.99 & 0.92 & 0.81 & 0.66 & 0.54 \\ [-1.7ex]
              & 12 & 3 & 2 & 0.99 & 0.99 & 0.91 & 0.87 & 1.08& & 12 & 6 & 2 & 0.99 & 0.96 & 0.88 & 0.81 & 0.75 \\ [-1.7ex]
              & 12 & 6 & 1 & 1.00 & 1.03 & 1.13 & 1.26 & 2.12& & 12 & 12& 1 & 1.00 & 1.00 & 1.00 & 0.98 & 0.99 \\ 
%              & 18 & 2 & 3 & 0.99 & 0.96 & 0.85 & 0.73 & 0.66& & 18 & 2 & 6 & 0.98 & 0.90 & 0.69 & 0.47 & 0.24 \\ 
%              & 18 & 3 & 2 & 1.00 & 0.98 & 0.95 & 0.97 & 1.39& & 18 & 3 & 4 & 1.00 & 0.92 & 0.78 & 0.61 & 0.50 \\ 
%              & 18 & 6 & 1 & 1.01 & 1.07 & 1.18 & 1.41 & 3.35& & 18 & 6 & 2 & 1.00 & 1.00 & 0.98 & 0.94 & 1.24 \\ 
   \hline
\end{tabular}
\label{ta:imp-dlfixtab}}}
\end{table}

% latex table generated in R 3.0.2 by xtable 1.7-1 package
% Thu Jan  8 16:54:24 2015
\begin{table}[h!]
\caption{{\footnotesize{\Bc{Values of $RE_2$ to compare the FI content of imperfect PROS$(n,S)$ with imperfect RSS of a fixed set size 6.}}}}
\vspace{0.3cm} % title name of the table
\centering % centering table
\footnotesize{\Bc{\begin{tabular}{llllccccccccccc}\hline
&         &&\multicolumn{11}{c}{$p$}      \\
\cline{5-15}
Distribution  &$S$ &$n$&$N$&0    &0.1   &0.2   &0.3   &0.4   &0.5   &0.6   &0.7   &0.8   &0.9   &1      \\\hline
Normal        &  4 & 2 & 3 &1.75 & 1.49 & 1.25 & 1.03 & 0.83 & 0.67 & 0.56 & 0.48 & 0.42 & 0.38 & 0.37   \\ [-1.7ex]
              & 6  & 2 & 3 &2.03 & 1.63 & 1.33 & 1.05 & 0.83 & 0.67 & 0.56 & 0.49 & 0.44 & 0.42 & 0.43  \\ [-1.7ex]
              & 6  & 3 & 2 &1.49 & 1.24 & 1.07 & 0.93 & 0.82 & 0.74 & 0.69 & 0.66 & 0.64 & 0.63 & 0.65  \\ [-1.7ex]
              & 6  & 6 & 1 &1.00 & 1.00 & 1.00 & 1.00 & 1.00 & 1.00 & 1.00 & 1.00 & 1.00 & 1.00 & 1.00  \\ [-1.7ex]
              &  8 & 2 & 3 &2.23 & 1.73 & 1.38 & 1.07 & 0.84 & 0.67 & 0.56 & 0.50 & 0.46 & 0.45 & 0.47  \\ [-1.7ex]
              & 12 & 2 & 3 &2.59 & 1.92 & 1.48 & 1.11 & 0.85 & 0.67 & 0.57 & 0.52 & 0.49 & 0.49 & 0.54  \\ [-1.7ex]
              & 12 & 3 & 2 &2.09 & 1.46 & 1.13 & 0.93 & 0.83 & 0.79 & 0.80 & 0.82 & 0.86 & 0.91 & 1.01  \\ [-1.7ex]
              & 12 & 6 & 1 &1.22 & 1.03 & 1.01 & 1.08 & 1.20 & 1.36 & 1.51 & 1.66 & 1.80 & 1.93 & 2.05  \\
%             & 18 & 2 & 3 &2.99 & 2.13 & 1.59 & 1.16 & 0.86 & 0.67 & 0.58 & 0.54 & 0.53 & 0.55 & 0.63  \\ [-1.7ex]
  %           & 18 & 3 & 2 &2.62 & 1.64 & 1.18 & 0.94 & 0.84 & 0.83 & 0.88 & 0.95 & 1.03 & 1.13 & 1.29  \\ [-1.7ex]
   %          & 18 & 6 & 1 &1.47 & 1.05 & 1.01 & 1.16 & 1.42 & 1.73 & 2.06 & 2.39 & 2.69 & 2.97 & 3.25  \\
    \hline
Exponential   &  4 & 2 & 3 &1.51 & 1.34 & 1.18 & 1.03 & 0.89 & 0.77 & 0.69 & 0.63 & 0.59 & 0.56 & 0.56  \\ [-1.7ex]
              & 6  & 2 & 3 &1.71 & 1.44 & 1.23 & 1.05 & 0.89 & 0.77 & 0.69 & 0.64 & 0.61 & 0.61 & 0.64  \\ [-1.7ex]
              & 6  & 3 & 2 &1.33 & 1.17 & 1.05 & 0.95 & 0.88 & 0.83 & 0.80 & 0.79 & 0.79 & 0.80 & 0.82  \\ [-1.7ex]
              & 6  & 6 & 1 &1.00 & 1.00 & 1.00 & 1.00 & 1.00 & 1.00 & 1.00 & 1.00 & 1.00 & 1.00 & 1.00  \\ [-1.7ex]
              &  8 & 2 & 3 &1.89 & 1.54 & 1.28 & 1.07 & 0.90 & 0.77 & 0.70 & 0.65 & 0.64 & 0.65 & 0.70  \\ [-1.7ex]
              & 12 & 2 & 3 &2.11 & 1.65 & 1.34 & 1.09 & 0.90 & 0.77 & 0.70 & 0.67 & 0.66 & 0.69 & 0.78  \\ [-1.7ex]
              & 12 & 3 & 2 &1.66 & 1.30 & 1.09 & 0.96 & 0.89 & 0.87 & 0.88 & 0.91 & 0.96 & 1.02 & 1.11  \\ [-1.7ex]
              & 12 & 6 & 1 &1.14 & 1.02 & 1.00 & 1.05 & 1.14 & 1.24 & 1.34 & 1.42 & 1.50 & 1.56 & 1.62  \\ 
 %            & 18 & 2 & 3 &2.47 & 1.85 & 1.44 & 1.13 & 0.91 & 0.77 & 0.71 & 0.69 & 0.71 & 0.78 & 0.92  \\ [-1.7ex]
  %           & 18 & 3 & 2 &1.98 & 1.42 & 1.12 & 0.96 & 0.90 & 0.90 & 0.95 & 1.02 & 1.11 & 1.21 & 1.35 \\ [-1.7ex]
   %          & 18 & 6 & 1 &1.28 & 1.04 & 1.01 & 1.11 & 1.27 & 1.46 & 1.64 & 1.79 & 1.93 & 2.04 & 2.13  \\ 
   \hline
Logistic      &  4 & 2 & 3 &1.89 & 1.57 & 1.30 & 1.04 & 0.83 & 0.66 & 0.55 & 0.47 & 0.42 & 0.38 & 0.37  \\ [-1.7ex]
              & 6  & 2 & 3 &2.25 & 1.73 & 1.38 & 1.07 & 0.83 & 0.66 & 0.55 & 0.48 & 0.44 & 0.42 & 0.44  \\ [-1.7ex]
              & 6  & 3 & 2 &1.55 & 1.27 & 1.08 & 0.93 & 0.82 & 0.74 & 0.69 & 0.67 & 0.65 & 0.65 & 0.67  \\ [-1.7ex]
              & 6  & 6 & 1 &1.00 & 1.00 & 1.00 & 1.00 & 1.00 & 1.00 & 1.00 & 1.00 & 1.00 & 1.00 & 1.00  \\ [-1.7ex]
              &  8 & 2 & 3 &2.50 & 1.87 & 1.45 & 1.10 & 0.84 & 0.66 & 0.56 & 0.49 & 0.46 & 0.46 & 0.49  \\ [-1.7ex]
              & 12 & 2 & 3 &2.94 & 2.08 & 1.56 & 1.14 & 0.85 & 0.66 & 0.56 & 0.51 & 0.50 & 0.51 & 0.58  \\ [-1.7ex]
              & 12 & 3 & 2 &2.23 & 1.51 & 1.14 & 0.93 & 0.83 & 0.80 & 0.81 & 0.84 & 0.89 & 0.96 & 1.07  \\ [-1.7ex]
              & 12 & 6 & 1 &1.23 & 1.03 & 1.01 & 1.09 & 1.22 & 1.39 & 1.57 & 1.73 & 1.89 & 2.03 & 2.17  \\ 
%             & 18 & 2 & 3 &3.44 & 2.34 & 1.69 & 1.19 & 0.86 & 0.66 & 0.57 & 0.54 & 0.54 & 0.57 & 0.68   \\ [-1.7ex]
 %            & 18 & 3 & 2 &2.89 & 1.72 & 1.20 & 0.93 & 0.84 & 0.84 & 0.90 & 0.99 & 1.09 & 1.21 & 1.42  \\ [-1.7ex]
  %           & 18 & 6 & 1 &1.50 & 1.06 & 1.01 & 1.17 & 1.45 & 1.80 & 2.16 & 2.52 & 2.86 & 3.17 & 3.48  \\ 
  \hline
\end{tabular}
\label{ta:imp-pros-rss-6}}}
\end{table}

% latex table generated in R 3.0.2 by xtable 1.7-1 package
% Thu Jan  8 16:54:24 2015
\begin{table}[h!]
\caption{{\footnotesize{\Bc{Values of $RE_2$ to compare the FI content of imperfect PROS$(n,S)$ with imperfect RSS of a fixed set size 12.}}}}
\vspace{0.3cm} % title name of the table
\centering % centering table
\footnotesize{\Bc{\begin{tabular}{llllccccccccccc}\hline
&         &&\multicolumn{11}{c}{$p$}      \\
\cline{5-15}
Distribution  &$S$&$n$&$N$&0    &0.1   &0.2   &0.3   &0.4   &0.5   &0.6   &0.7   &0.8   &0.9   &1      \\\hline
Normal        & 6 & 2 & 6 &2.24 & 1.67 & 1.17 & 0.78 & 0.52 & 0.36 & 0.27 & 0.21 & 0.18 & 0.16 & 0.16   \\ [-1.7ex]
              & 6 & 3 & 4 &1.62 & 1.27 & 0.94 & 0.68 & 0.51 & 0.40 & 0.33 & 0.29 & 0.26 & 0.24 & 0.24  \\ [-1.7ex]
              & 12& 2 & 6 &2.86 & 1.98 & 1.31 & 0.82 & 0.52 & 0.36 & 0.27 & 0.22 & 0.20 & 0.19 & 0.21  \\ [-1.7ex]
              & 12& 3 & 4 &2.27 & 1.49 & 1.00 & 0.69 & 0.51 & 0.43 & 0.38 & 0.36 & 0.35 & 0.35 & 0.37  \\ [-1.7ex]
              & 12& 4 & 3 &1.78 & 1.23 & 0.89 & 0.69 & 0.59 & 0.53 & 0.50 & 0.49 & 0.49 & 0.50 & 0.52  \\ [-1.7ex]
              & 12& 6 & 2 &1.31 & 1.05 & 0.88 & 0.79 & 0.74 & 0.72 & 0.71 & 0.70 & 0.71 & 0.71 & 0.73  \\ [-1.7ex]
              & 12& 12& 1 &1.00 & 1.00 & 1.00 & 1.00 & 1.00 & 1.00 & 1.00 & 0.99 & 0.99 & 0.99 & 0.99  \\ 
    %         & 18& 2 & 6 &3.29 & 2.20 & 1.41 & 0.85 & 0.53 & 0.36 & 0.27 & 0.23 & 0.22 & 0.21 & 0.24  \\ [-1.7ex]
   %          & 18& 3 & 4 &2.83 & 1.66 & 1.03 & 0.69 & 0.52 & 0.44 & 0.42 & 0.41 & 0.41 & 0.43 & 0.47  \\ [-1.7ex]
 %            & 18& 6 & 2 &1.58 & 1.07 & 0.89 & 0.85 & 0.87 & 0.91 & 0.96 & 1.01 & 1.06 & 1.10 & 1.16  \\ 
 \hline
Exponential   & 6 & 2 & 6 &1.80 & 1.46 & 1.14 & 0.86 & 0.66 & 0.53 & 0.44 & 0.39 & 0.36 & 0.35 & 0.36  \\ [-1.7ex]
              & 6 & 3 & 4 &1.38 & 1.18 & 0.97 & 0.79 & 0.65 & 0.56 & 0.51 & 0.47 & 0.46 & 0.45 & 0.46  \\ [-1.7ex]
              & 12& 2 & 6 &2.23 & 1.68 & 1.24 & 0.90 & 0.67 & 0.53 & 0.44 & 0.40 & 0.39 & 0.40 & 0.44  \\ [-1.7ex]
              & 12& 3 & 4 &1.78 & 1.33 & 1.01 & 0.79 & 0.66 & 0.59 & 0.56 & 0.56 & 0.57 & 0.59 & 0.64  \\ [-1.7ex]
              & 12& 4 & 3 &1.46 & 1.15 & 0.93 & 0.80 & 0.72 & 0.69 & 0.68 & 0.68 & 0.70 & 0.72 & 0.75  \\ [-1.7ex]
              & 12& 6 & 2 &1.19 & 1.03 & 0.93 & 0.87 & 0.84 & 0.83 & 0.84 & 0.85 & 0.86 & 0.87 & 0.89  \\ [-1.7ex]
              & 12& 12& 1 &1.00 & 1.00 & 1.00 & 1.00 & 1.00 & 1.00 & 1.00 & 1.00 & 1.00 & 1.00 & 1.00  \\ 
%             & 18& 2 & 6 &2.58 & 1.86 & 1.33 & 0.93 & 0.68 & 0.53 & 0.45 & 0.42 & 0.41 & 0.44 & 0.51  \\ [-1.7ex]
  %           & 18& 3 & 4 &2.10 & 1.44 & 1.04 & 0.79 & 0.66 & 0.61 & 0.60 & 0.62 & 0.65 & 0.69 & 0.76  \\ [-1.7ex]
   %          & 18& 6 & 2 &1.34 & 1.05 & 0.93 & 0.91 & 0.93 & 0.98 & 1.02 & 1.06 & 1.10 & 1.14 & 1.17  \\
    \hline
Logistic      & 6 & 2 & 6 &2.41 & 1.76 & 1.21 & 0.78 & 0.51 & 0.35 & 0.25 & 0.20 & 0.17 & 0.16 & 0.16  \\ [-1.7ex]
              & 6 & 3 & 4 &1.69 & 1.30 & 0.95 & 0.67 & 0.50 & 0.39 & 0.32 & 0.28 & 0.26 & 0.24 & 0.24   \\ [-1.7ex]
              & 12& 2 & 6 &3.14 & 2.11 & 1.36 & 0.83 & 0.51 & 0.35 & 0.26 & 0.22 & 0.19 & 0.19 & 0.20  \\ [-1.7ex]
              & 12& 3 & 4 &2.43 & 1.55 & 1.00 & 0.68 & 0.50 & 0.42 & 0.37 & 0.36 & 0.35 & 0.36 & 0.38  \\ [-1.7ex]
              & 12& 4 & 3 &1.87 & 1.25 & 0.89 & 0.69 & 0.58 & 0.53 & 0.51 & 0.50 & 0.51 & 0.52 & 0.55  \\ [-1.7ex]
              & 12 & 6 & 2&1.33 & 1.05 & 0.88 & 0.79 & 0.74 & 0.72 & 0.71 & 0.72 & 0.72 & 0.73 & 0.75  \\ [-1.7ex]
              & 12 & 12& 1&1.00 & 1.00 & 1.00 & 1.00 & 1.00 & 1.00 & 1.00 & 1.00 & 1.00 & 1.00 & 1.00  \\ 
%             & 18 & 2 & 6&3.77 & 2.41 & 1.49 & 0.87 & 0.52 & 0.35 & 0.26 & 0.23 & 0.21 & 0.21 & 0.24  \\ [-1.7ex]
 %            & 18 & 3 & 4&3.19 & 1.77 & 1.05 & 0.68 & 0.51 & 0.44 & 0.42 & 0.42 & 0.43 & 0.46 & 0.52  \\ [-1.7ex]
  %           & 18 & 6 & 2&1.63 & 1.08 & 0.89 & 0.85 & 0.89 & 0.94 & 1.01 & 1.07 & 1.13 & 1.18 & 1.25  \\ 
  \hline
\end{tabular}
\label{ta:imp-pros-rss-12}}}
\end{table}

\section{Other Information Criteria} \label{sec:oic}
The concept of information and uncertainty of random samples  is so rich  that 
several  measures have been proposed  to study  different aspects of these concepts.
For example, in the  Engineering studies, the  Shannon 
entropy, \BL{  R\'enyi} entropy and KL information  measures are used more than FI to quantify 
the information and uncertainty structures of random samples.
These measures quantify the amount of uncertainty inherent in the joint probability distribution of a random  sample and  have been applied in many  areas such as   ecological studies, computer sciences and 
information technology,    in \Bc{different  contexts}  including  order statistics, spacings, censored data, 
reliability, life testing, \Bc{record data} and text analysis. For more details see \cite{Joza-Ahma-2013} and \cite{John-2004} and  references therein. 

In this section, we compare the Shannon entropy, R\'enyi  entropy and KL information of PROS
data with SRS and RSS  data of the same size. Throughout this section,  the  
 subsetting process of  PROS design  and  the ranking process 
of  RSS  are assumed  to be perfect.

%----------------------------------------------------------------------------------------
% \subsection{Entropy}
%----------------------------------------------------------------------------------------
\subsection{Shannon Entropy of the PROS sample}\label{sub:ent}
Let $X$ be a continuous random variable with pdf $f(\cdot,\dt)$. 
The Shannon entropy associated with  $X$,
 is defined as 
\[
H(X; \dt)=-\int f(x;\dt) \log f(x;\dt) dx,
\]
subject to the existence of the integral. 
\Bc{The
Shannon entropy, as a quantitative measure of information (uncertainty), is
extensively used in information technology, computer science and other engineering
fields.}
%The Shannon entropy, as a quantitative measure of information (uncertainty) is extensively used in 
% information technology and computer science and other engineering fields.
In practice, smaller values of the Shannon entropy are more desirable  \citep[see][]{John-2004}.
The Shannon entropy content of a SRS of size $n$ is given by
\[
H_n({\bf X}_{srs}; \dt)=-\sum_{i=1}^{n}\int f(x;\dt)\log f(x;\dt) dx
=n\, H(X_1; \dt).
\]
Similarly, for an RSS of size $n$ (with the  set size $n$)
\[
H_n({\bf X}_{rss}; \dt)=-\sum_{i=1}^{n}\int f^{(i:n)}(x;\dt)\log f^{(i:n)}(x;\dt) dx,
\]
where $f^{(i:n)}(\cdot;\dt)$ is the pdf of the $i$-th order statistic in a SRS of size $n$ from $f(\cdot; \dt)$. Furthermore, for a \BL{ PROS($n, S$)} sample,  it is easy to see that
\[
H_n({\bf X}_{pros}; \dt)=-\sum_{r=1}^{n}\int f_{(d_r)}(y;\dt) \log f_{(d_r)}(y;\dt) dy.
\]
In the following lemma,  we show that the Shannon entropy of PROS data is smaller than 
that of SRS data of the same size. Unfortunately,  we were not able to obtain an ordering relationship  among  the Shannon entropy of RSS and PROS data of the same size. Instead,  we obtain a  lower bound for   the Shannon entropy of a \BL{  PROS($n, S$)} sample in terms of  the Shannon entropy of an RSS data of size $S$ when the set size is $S$.  

\begin{lemma}\label{le:sh-pros-srs}
Let ${\bf X}_{pros}$ be a \BL{ {\rm PROS($n, S$)}} sample  from a population with pdf $f(\cdot; \dt)$ and let $m=S/ n$ be   the number of observations in each subset. Suppose ${\bf X}_{srs}$ is a SRS of size $n$ from $f(\cdot; \dt)$  with the Shannon  entropy   $H_n({\bf X}_{srs}; \dt)$ and  $H_S({\bf X}_{rss}; \dt)$ 
 represent the Shannon  entropy of  an RSS of size $S$ when the set  size is  $S$. Then, 
$$\frac{1}{m}H_S({\bf X}_{rss}; \dt)\le H_n({\bf X}_{pros}; \dt)\le H_n({\bf X}_{srs}; \dt),\quad \text{ for all $n\in N$}. 
$$
\end{lemma}

\begin{proof} Using \eqref{f-x} and  convexity of  $h(t)=t\log t , t>0$, we have 
\begin{eqnarray*}
H_n({\bf X}_{pros}; \dt) 
%&=& - \sum_{r=1}^{n} \int f_{(d_r)}(x;\dt) \log f_{(d_r)}(x;\dt) dx \\
&\le&  -n \int \left( \frac{1}{n} \sum_{r=1}^{n} f_{(d_r)}(x;\dt) \right)
\left( \log\left[
\frac{1}{n} \sum_{r=1}^{n} f_{(d_r)}(x;\dt)
\right]\right) dx \\
%&=& -n \int \left( \frac{1}{S} \sum_{u=1}^{S} f^{(u:S)}(x;\dt) \right)
%\left( \log\left[
%\frac{1}{S} \sum_{r=1}^{n} f^{(u:S)}(x;\dt)
%\right]\right) dx \\
%&=& - n \int f(x;\dt) \log f(x;\dt) dx \\
&=&   H_n({\bf X}_{srs}; \dt).
\end{eqnarray*}
Furthermore, using \eqref{f-per-x} and convexity of  $h(t)=t\log t, t>0$,  we have
\begin{eqnarray*}
H_n({\bf X}_{pros}; \dt) 
%&=& \sum_{r=1}^{n} H(X_{[d_r]}) \\
%&=&- \sum_{r=1}^{n} \int f_{(d_r)}(x;\dt) \log f_{(d_r)}(x;\dt) dx \\
&=& - \sum_{r=1}^{n} \int 
\left(
\frac{1}{m} \sum_{u \in d_r} f^{(u:S)}(x;\dt)
\right) 
\left( \log 
\left[
\frac{1}{m} \sum_{u \in d_r} f^{(u:S)}(x;\dt)
\right]\right) dx \\
&\ge& -\frac{1}{m} \sum_{r=1}^{n} \sum_{u \in d_r}
\int f^{(u:S)}(x;\dt) \log f^{(u:S)}(x;\dt) dx \\
%&=& -\frac{1}{m} \sum_{v=1}^{S} 
%\int f^{(v:S)}(x;\dt) \log f^{(v:S)}(x;\dt) dx \\
%&=& \frac{1}{m} \sum_{u=1}^{S} H(X_{(u)}) \\
&=& \frac{1}{m}  H_S({\bf X}_{rss}; \dt), 
\end{eqnarray*}
which completes the proof.
\end{proof}
\subsection{R\'enyi entropy of PROS data}\label{sub:ren}
In this section we use the \BL{  R\'enyi} entropy as \Bc{a quantitative measure of the entropy} associated with PROS data ${\bf X}_{pros}$.
The \BL{  R\'enyi} entropy  of a random variable $X$ with pdf $f(\cdot; \dt)$ is defined as follows  
\[
H_{\alpha}(X; \dt)=
\frac{1}{1-\alpha} \log \E[f^{\alpha-1}(X; \dt)],
\]
where $\alpha > 0,\alpha \ne 1$.
The \BL{  R\'enyi} entropy is a very   general measure  and includes    the  Shannon entropy  as its   special case  due to the following relationship
\[
\lim_{\alpha \to 1} H_{\alpha}(X; \dt) = - \int f(x; \dt) \log f(x; \dt) dx = H(X; \dt).
\]
Due to the  flexibility  of  the \BL{  R\'enyi} entropy, 
$H_{\alpha}(X; \dt)$ has been used in many fields such as
 statistics, ecology, engineering and etc. 
We derive the \BL{  R\'enyi} entropy  of ${\bf X}_{pros}$ and  compare it with the \BL{  R\'enyi} entropy of  ${\bf X}_{srs}$.  
We present the results for  $0< \alpha < 1$ and the  
  case  with $\alpha >1$,  which requires further investigation,  will be presented in later works.
To this end, the \BL{  R\'enyi} entropy
 of a SRS of size $n$ is given by
 \[
 H_{\alpha, n}({\bf X}_{srs}; \dt)= \frac{1}{1-\alpha} \sum_{i=1}^{n} \log \int f^{\alpha}(x_i;\dt)\, dx_i 
 =n\,H_{\alpha}(X_1; \dt );
 \]
 and for  an RSS with set size $n$, 
 \[
 H_{\alpha, n}({\bf X}_{rss}; \dt)= \frac{1}{1-\alpha} \sum_{i=1}^{n} \log \int [f^{(i:n)}(x;\dt)]^{\alpha} dx.
 \]
  Also, for a \BL{ PROS($n, S)$} sample, one gets
 \[
 H_{\alpha, n}({\bf X}_{pros}; \dt)= \frac{1}{1-\alpha} \sum_{r=1}^{n} \log \int [f_{(d_r)}(x;\dt)]^{\alpha} dx.
 \]

 \begin{lemma}\label{le:rey}
 Let $H_{\alpha, n}({\bf X}_{pros}; \dt)$ represent the \BL{  R\'enyi} entropy of  a \BL{ {\rm PROS($n, S$)}}  sample of size $n$  from a population with pdf $f(\cdot; \dt)$.  Suppose ${\bf X}_{srs}$ and ${\bf X}^*_{rss}$ be  a SRS  of size $n$ and \Bc{an RSS of size $S$} (with the set size $S$) from $f(\cdot; \dt)$, respectively. For any $0<\alpha <1$ and all $n \in N$, we have
 \[
\frac{1}{m}H_{\alpha, S}({\bf X}^*_{rss}; \dt)\le H_{\alpha, n}({\bf X}_{pros}; \dt) \le H_{\alpha, n}({\bf X}_{srs}; \dt).
 \]
% where $H_{\alpha}(X_{SRS})$ is the Reyni information associated with simple random sample of the 
% same size.
 \end{lemma}
\begin{proof} %To show the result, we used \eqref{f-per-x} for the first equality.
  By using \eqref{f-per-x} and the concavity of the functions  
  $h_1(t)=\log t$ and $h_2(t)=t^{\alpha}$, we have  % The last equality is from \eqref{f-x}.
% To show the results, we consider two cases i.e. $0<\alpha<1$ and $\alpha>1$.
% First, we show the results when $0< \alpha <1$.
\begin{eqnarray*}
H_{\alpha, n}({\bf X}_{pros}; \dt) 
&\le & \frac{n}{1-\alpha} 
\left[
\log \int \frac{1}{n} \sum_{r=1}^{n} 
\left(
\frac{1}{m} \sum_{u \in d_r} f^{(u:S)}(x;\dt)
\right)^{\alpha} dx
\right]\\
&\le & \frac{n}{1-\alpha} 
\log \int 
\left(
\frac{1}{S} \sum_{r=1}^{n} \sum_{u \in d_r} f^{(u:S)}(x;\dt)
\right)^{\alpha} dx \\
%&=& \frac{n}{1-\alpha} \log \int [f(x;\dt)]^{\alpha} dx \\
&=&  H_{\alpha, n}({\bf X}_{srs}; \dt).  
\end{eqnarray*}
Similarly, one can show the following inequalities 
\begin{eqnarray*}
H_{\alpha, n}({\bf X}_{pros}; \dt) 
%&=& 
%\frac{1}{1-\alpha} \sum_{r=1}^{n} \log \int 
%\left(
%\frac{1}{m} \sum_{u \in d_r} f^{(u:S)}(x;\dt)
%\right)^{\alpha} dx \\
&\ge & \frac{1}{1-\alpha} \sum_{r=1}^{n} \log 
\left(
\frac{1}{m} \sum_{u \in d_r}  \int [f^{(u:S)}(x;\dt)]^{\alpha} dx
\right)\\
&\ge & \frac{1}{m(1-\alpha)} \sum_{r=1}^{n} \sum_{u \in d_r} \log 
\left(
  \int [f^{(u:S)}(x;\dt)]^{\alpha} dx
\right)\\
%&= & \frac{1}{m(1-\alpha)} \sum_{v=1}^{S} \log 
%\left(
%  \int [f^{(v:S)}(x;\dt)]^{\alpha} dx
%\right)\\
&=& \frac{1}{m} H_{\alpha, S}({\bf X}^*_{rss}; \dt), 
\end{eqnarray*}
 which complete the proof.
\end{proof}
%Hence when $\alpha < 1$, we have 
%\begin{eqnarray}\label{re-l-rss}
%H_{\alpha}({\bf X}_{PROS}^{n}) \ge \frac{1}{m} H_{\alpha}({\bf X}_{RSS}^{S}) 
%\end{eqnarray}
%Therefore for $\alpha < 1$ from \eqref{re-l-srs} and \eqref{re-l-rss}, we have
%\begin{eqnarray}\label{re-l}
%\frac{1}{m} H_{\alpha}({\bf X}_{RSS}^{S}) \le 
%H_{\alpha}({\bf X}_{PROS}^{n}) \le
%H_{\alpha}({\bf X}_{SRS}^{n})
%\end{eqnarray} 
%----------------------------------------------------------------------------------------
% \subsection{Kullback-Leibler Information}
%----------------------------------------------------------------------------------------
\subsection{KL Information of the PROS technique}\label{sub:kul}

The Kullback-Leibler (KL) discrepancy  is another 
measure  which can be used to quantify the information regarding a random phenomenon by comparing two 
probability density functions corresponding to a random experiment.
Consider  two  pdfs $f(\cdot; \dt)$ and $g(\cdot; \dt)$. The KL information measure based on $f(\cdot; \dt)$ and $g(\cdot; \dt)$ 
is defined by
\begin{eqnarray*}
K(f, g)=\int f(t;\dt) \log 
\left(
\frac{f(t;\dt)}{g(t;\dt)}
\right) dt,
\end{eqnarray*} 
which quantifies  the information lost  by  
 using $g(\cdot;\dt)$ for the density of the random variable  $X$ instead of $f(\cdot;\dt)$. 
In this section,  using the  KL measure we   make a comparison among PROS sampling,  simple random sampling  
and ranked set sampling  designs 
to determine which design provides 
more informative samples from   the underlying population. 
%While \cite{Joza-Ahma-2013} focused on KL distance for comparison between RSS and SRS as 
%\begin{eqnarray*}
%K\left(L_{srs}^n(\dt|{\bf y}),L_{rss}^n(\dt|{\bf y})\right)=\oint L_{srs}^n(\dt|{\bf y}) \log 
%\left(
%\frac{L_{srs}^n(\dt|{\bf y})}{L_{rss}^n(\dt|{\bf y})}
%\right) d{\bf y};
%\end{eqnarray*}
To this end, we use 
\begin{eqnarray}\label{kl-lik}
K\left(L_{pros}(\dt|{\bf y}),L_{srs}(\dt|{\bf y})\right)=\oint L_{pros}(\dt|{\bf y}) \log 
\left(
\frac{L_{pros}(\dt|{\bf y})}{L_{srs}(\dt|{\bf y})}
\right) d{\bf y},
\end{eqnarray}
 to   compare  \BL{ PROS($n, S$)} and simple random sampling designs, where $L_{pros}(\dt|{\bf y})$ and $L_{srs} \dt|{\bf y})$ denote the likelihood functions of PROS and SRS data of the same size, respectively.
The KL information measure \Bc{for comparing  ranked set sampling and simple random sampling}  is defined similarly by using \eqref{kl-lik} and setting $S=n$ in PROS sampling design.    One  can interpret \eqref{kl-lik} in terms of a 
hypothesis testing problem  within the Neyman-Pearson log-likelihood ratio  testing  framework \citep[see][]{John-2004}.
%to test  the hypothesis  $H_0:$ The likelihood function $L_{pros}(\dt|{\bf y})$ under PROS sampling design against an alternative hypothesis  
% $H_1:$  the likelihood function $L_{srs}(\dt|{\bf y})$ under simple random sampling design.
%To this end, it is custom in this framework that
%we focus on the log-likelihood ratio 
%$
%\log 
%\left(
%\frac{L_{pros}^n(\dt|{\bf y})}{L_{srs}^n(\dt|{\bf y})}
%\right)
%$
% as the test statistic whose large values support the PROS sampling design. It is interesting that KL 
%information discrimination \eqref{kl-lik} can be considered as expectation of the test statistic under $H_0$,
%
%To this end, it is cos
%between the likelihood functions corresponding to the sampling designs.
%To this end, we are KL information discrimination to make comparison between PROS sampling and simple 
%random sampling based on a fixed  and arbitrary sample for  two sampling designs as follows
%%To make a comparison between the 
%%sample arisen from PROS sampling design $x_{pros}$ and simple random sample $x_{srs}$ of the same size, 
%%we use the KL information discrimination  as follows

%where $L_{pros}^n(\dt|{\bf y})$ and $L_{srs}^n(\dt|{\bf y})$ denote the likelihood functions of the sample ${\bf y}$ 
%under PROS and SRS designs of sizes n, respectively.
\begin{lemma}\label{le:kl-simp}
Let $L_{pros}(\dt|{\bf y})$ and $L_{srs}(\dt|{\bf y})$ denote,   respectively,  the likelihood functions of a  \BL{ {\rm PROS($n, S$)}} sample   and a  SRS  of size $n$ from a population with pdf $f(\cdot; \dt)$. Then we have
\[
K\left(L_{pros}(\dt|{\bf y}), L_{srs}(\dt|{\bf y})\right)=\sum_{r=1}^{n} \int
f_{(d_r)}(y;\dt) \log
\left(
\frac{f_{(d_r)}(y;\dt)}{f(y;\dt)}
\right) dy.
\]
\end{lemma}
\begin{proof} To show the result,  using  \eqref{kl-lik} we have 
\begin{align*}
K\left(L_{pros}(\dt|{\bf y}),L_{srs}(\dt|{\bf y})\right) 
%&= \oint L_{pros}^n(\dt) \log 
%\left(
%\frac{L_{pros}^n(\dt)}{L_{srs}^n(\dt)}
%\right) d{\bf y} \\
&= \sum_{r=1}^{n}\oint
\left\{
\prod_{h=1}^{n} f_{(d_h)}(y_h;\dt) 
\right\}
\log 
\left(
\frac{f_{(d_r)}(y_r;\dt)}{f(y_r;\dt)}
\right)
\left\{
\prod_{j=1}^{n} dy_j 
\right\} \\
&= \sum_{r=1}^{n}\int 
 f_{(d_r)}(y;\dt) 
\log 
\left(
\frac{f_{(d_r)}(y;\dt)}{f(y;\dt)}
\right)
 dy;
\end{align*}
where the  last equality   follows from the independence  of  observations  and the fact that   $n-1$ of the integrals are 1. \end{proof}

\noindent In the following lemma, we show that KL information distance between the 
likelihoods of  PROS and  SRS sampling designs is greater than the one between 
the likelihoods of two SRS sampling designs. Hence, PROS data are more informative than SRS data about the underlying population.  We also obtain a lower bound for the KL information  between 
the likelihoods of PROS and SRS data of the same size.
\begin{lemma}\label{le:kl-pros-srs} Let $L_{pros}(\dt|{\bf y})$ 
denote the likelihood function  of a \BL{ {\rm PROS($n, S$)}} sample  from a population with pdf $f(\cdot, \dt)$. Suppose  $L_{srs,1}(\dt|{\bf y})$ and $L_{srs,2}(\dt|{\bf y})$ denote 
 the  likelihood functions  of  simple random samples of size $n$ from  $f(\cdot;\dt)$ and $g(\cdot;\dt)$, respectively.  In addition, let $L_{rss^*}(\dt| {\bf y})$  represent the likelihood function of a RSS of size $S$ when the set size is $S$. Then,
\[
K\left(L_{srs,1}(\dt|{\bf y}),L_{srs,2}(\dt|{\bf y})\right) \leq K\left(L_{pros}(\dt|{\bf y}),L_{srs,2}(\dt|{\bf y})\right)\leq  \frac{1}{m} K
  \left(\tilde{L}_{rss^*}(\dt|{\bf y}),L_{srs,2}(\dt|{\bf y})
  \right).  
\] 
\end{lemma}
\begin{proof} Applying Lemma \ref{le:kl-simp} and  using the convexity  of  $h(t)=t\log t$, $t>0$,  we derive
\begin{eqnarray*}
K\left(L_{pros}(\dt | {\bf y}),L_{srs,2}(\dt | {\bf y})\right) 
%&=& \sum_{r=1}^{n} \int f_{(d_r)}(y;\dt) 
%\log\left(
%\frac{f_{(d_r)}(y;\dt)}{g(y;\dt)} 
%\right) dy \\
&=& \sum_{r=1}^{n} \int g(y;\dt) 
\left(
\frac{f_{(d_r)}(y;\dt)}{g(y;\dt)}
\right)
\log 
\left(
\frac{f_{(d_r)}(x;\dt)}{g(y;\dt)} 
\right) dy \\
& \ge & n \int g(y;\dt) 
\left[
\frac{1}{n}\sum_{r=1}^{n} 
\frac{f_{(d_r)}(y;\dt)}{g(y;\dt)}
\right]
\log 
\left[
\frac{\frac{1}{n} \sum_{r=1}^{n}f_{(d_r)}(y;\dt)}{g(y;\dt)} 
\right] dy \\
&=& n \int f(y;\dt) \log 
\left(
\frac{f(y;\dt)}{g(y;\dt)}
\right) dy \\
&=& K\left(L_{srs,1}(\dt),L_{srs,2}(\dt)\right),
\end{eqnarray*}
which shows  the first inequality.
Similarly, 
\begin{align*}
K
  \left(L_{pros}(\dt|{\bf y}),L_{srs,2}(\dt|{\bf y})
  \right) 
%  &= 
%  \sum_{r=1}^{n} \int f_{(d_r)}(y;\dt) 
%\left(
%\frac{f_{(d_r)}(y;\dt)}{g(y;\dt)} 
%\right) dy \\
%&=\sum_{r=1}^{n} \int g(y;\dt) 
%\left(
%\frac{f_{(d_r)}(y;\dt)}{g(y;\dt)}
%\right)
%\log 
%\left(
%\frac{f_{(d_r)}(x;\dt)}{g(y;\dt)} 
%\right) dy \\
&= \sum_{r=1}^{n} \int g(y;\dt)  
\left( 
\frac{1}{m} \sum_{u\in d_r}
\frac{f^{(u:S)}(y;\dt)}{g(y;\dt)}
\right)
\log 
\left(
\frac{1}{m} \sum_{u\in d_r}
\frac{f^{(u:S)}(y;\dt)}{g(y;\dt)} 
\right) dy \\
%& \le \frac{1}{m}  \sum_{r=1}^{n} \sum_{u\in d_r} \int f^{(u:S)}{(d_r)}(y;\dt)
%\log 
%\left(
%\frac{f^{(u:S)}(x;\dt)}{g(y;\dt)} 
%\right) dy \\
& \le \frac{1}{m}  \sum_{v=1}^{S}  \int f^{(v:S)}(y;\dt)
\log 
\left(
\frac{f^{(v:S)}(x;\dt)}{g(y;\dt)} 
\right) dy \\
&= \frac{1}{m} K
  \left({L}_{rss^*}(\dt|{\bf y}),L_{srs,2}(\dt|{\bf y})
  \right),
  \end{align*}
% where the inequality is implied by the convexity of the function $h(t)=t\log t$ and rewriting 
% the double summations as one summation over $v=1,\ldots,S$. 
which completes the proof.
%where the inequality is implied by the the convexity of the function $h(t)=t\log t,~t>0$; and the fourth equality is from 
%\[
%\frac{1}{n} \sum_{r=1}^{n} f_{(d_r)}(y;\dt) = \frac{1}{S} \sum_{r=1}^{n} \sum_{u \in d_r}
%f^{(u:S)}(y;\dt)= f(y;\dt).\]
\end{proof}
\section{Concluding Remarks}\label{concluding}
In this paper, we have considered the information content and uncertainty associated with PROS samples from a population.  First, we have compared the FI content of PROS samples with the FI content of SRS and RSS data of the same size under both perfect and imperfect  subsetting assumptions. We showed that \Bc{PROS sampling  design results in} more informative observations from the underlying population than simple random sampling and ranked set sampling.  Some \Ba{examples} are presented to show the amount of the extra information provided by PROS sampling design. We have then considered other information and uncertainty measures such as the Shannon entropy, R\'enyi entropy and the KL information measures. Similar results have been obtained under the perfect subsetting assumption. It would naturally be of interest to extend these results to imperfect subsetting situations. The results of this paper suggest  that one might be able to obtain more powerful tests for testing hypothesis or model selection  problems   based on PROS data. For example,  it seems promising to develop goodness of fit tests based on PROS data  under KL information measure. We believe that further investigation of PROS sampling design  under  the missing information criterion as in  \cite{Hatefi201316}  is of interest and appealing as well.

\section*{Acknowledgement}
%We would like to thank three anonymous reviewers for their constructive  comments and suggestions on the original version of this paper. 
This research has been done when Armin Hatefi was a PhD student at the University of Manitoba and his research was supported by the University of Manitoba Graduate Fellowship (UMGF) and 
Manitoba Graduate Scholarship (MGS). 
Mohammad Jafari Jozani gratefully acknowledges the research support of the NSERC Canada.
%%%%%%%%%%%%%%%%%%%%%%%%%%%%%%%%%%%%%%%%%%%%%%%%%%%%%%%%%%%%%%%%%%%%%%%%%%%%%%%%%%%%%%%%%%
\Bc{ \section*{Appendix:\\ FI of unbalanced PROS and the effect of misplacement errors}}
 
In this section, we study the FI matrix of the unbalanced PROS sampling design in a general setting when the subsets are allowed to be of different sizes.
 To obtain an unbalanced PROS sample , we first need to determine the  sample of size $K$ and set size $S$. Judgment sub-setting process is then applied to create $K$ sets. We group these $K$ sets into $N$ cycles $G_i=\{S_{1,i},\ldots,S_{n_i,i}\};\, i=1,\ldots,N$, where $\sum_{i=1}^{N}n_i=K$.
 Let $D_{r,i}=\{{d_{r[1]i}},\ldots,d_{r[n_i]i}\}$  be the design parameter associated with set $S_{r,i}$, where ${d_{r[l]i}}; l=1,\ldots,n_i$ is the $l$-th judgment subset in the set $S_{r,i}$. In each cycle $G_i; i=1,\ldots,N$, we randomly select a unit from one of the sets (particularly from the judgment subset $d_{r[r]i}; r=1,\ldots, n_i$) for full measurement, say $X_{[d_r]i}$ and the number of unranked units in subset $d_{r[r]i}$ is denoted by $m_{ri}; r=1,\ldots,n_i; i=1,\ldots,N$. To this end, the collection of measured observations $\{X_{[d_r]i};r=1,\ldots,n_i;i=1,\ldots,N\}$ is an unbalanced PROS sample of size $K=\sum_{i=1}^{N}n_i$. Table \ref{ta:upros-ex} illustrates the construction of an unbalanced PROS sample of size of $K=5$ with set size $S=6$ and cycle size $N=2$ so that in the first cycle we declare three subsets $n_1=3$   and  two subsets $n_2=2$ of different sizes in the first and second cycles, respectively. In each set, $m_{ri}$ represent the number of unranked units in the selected subset. For more details  see \cite {Oztu-11-ees}. 
  \begin{table}[h!]
\begin{center}
\caption{\small An example of unbalanced PROS design when $S=6, K=5, N=2, n_1=3, n_2=2$ and $m_{ri}$ represents size of the selected subset in each set.}
\vspace{0.5cm}
{\begin{tabular}{ccccc} \hline\hline
cycle & set &  Subsets & $m_{ri}$ &Observation \\ \hline
 1    & $S_{1,1}$ & $D_{1,1}=\{\mbold{d_{1[1]1}},d_{1[2]1},d_{1[3]1}\}=\{ \mbold{\{1,2,3\}},\{4,5\},\{6\}\}$  & 3 &$X_{[d_1]1}$  \\ 
      & $S_{2,1}$ & $D_{2,1}=\{d_{2[1]1},\mbold{d_{2[2]1}},d_{2[3]1}\}=\{ \{1,2,3\},\mbold{\{4,5\}},\{6\} \}$ & 2&$X_{[d_2]1}$  \\
      & $S_{3,1}$ & $D_{3,1}=\{d_{3[1]1},d_{3[2]1},\mbold{d_{3[3]1}}\}=\{ \{1,2,3\},\{4,5\},\mbold{\{6\}} \}$ & 1&$X_{[d_3]1}$  \\\hline
 2    & $S_{1,2}$ & $D_{1,2}=\{\mbold{d_{1[1]2}},d_{1[2]2}\}=\{ \mbold{\{1,2\}},\{3,4,5,6\}\} $               & 2&$X_{[d_1]2}$  \\
      & $S_{2,2}$ & $D_{2,2}=\{d_{2[1]2},\mbold{d_{2[2]2}}\}=\{ \{1,2\},\mbold{\{3,4,5,6\}}\} $               & 4&$X_{[d_2]2}$  \\\hline\hline
 \end{tabular}}
  \label{ta:upros-ex}
 \end{center}  
 \end{table}

\noindent We fist present the following result. 
 
\begin{lemma}\label{le-frm}
Let $Y_{ri}=X_{[d_r]i}$  be an observation from unbalanced PROS sampling design 
from  a continuous distribution with pdf $f(\cdot;\dt)$. With the knowledge of 
the design parameter $D_{r,i}$, the pdf of $Y_{ri}$ is given by

\[ f_{[r;m_{ri}]}(y;\dt)= \frac{1}{m_{ri}} \sum_{v\in d_{r[r]i}} f^{[v:S]}(y;\dt),
\]
where $f^{[v:S]}(y;\dt)$ is the pdf of the $v$-th judgment order statistics between $S$ data.
\end{lemma}
\begin{proof}
 For each  $Y_{ri}$  define the  latent vector  
$\bd^{[d_r]i}= \left(\ld^{[d_r]i}(v), v\in d_{r[r]i}\right)$,   where 
\[  
\ld^{[d_r]i}(v) = 
\left\{
       \begin{array}{ll}
       1 & \mbox{if  $Y_{ri}$ is selected from the $v$-th position within the subset $d_{r[r]i}$}; \\
       0 & \mbox{otherwise},
       \end{array} 
\right. 
\]
with  $\sum_{v \in d_{r[r]i}} \ld^{[d_r]i}(v)=1$. The joint pdf of  
 $(Y_{ri},\bd^{[d_r]i})$  is  given by 
\begin{eqnarray}\nonumber
f(y,\dd^{[d_r]i};\dt)= \prod_{r=1}^{n_i}\prod_{v \in d_{r[r]i}} 
\left\{\frac{1}{m_{ri}} f^{[v:S]}(y;\dt)\right\}^{\dd^{[d_r]i}(v)}.
\end{eqnarray}
%
%According to the fact that  $X_{(d_r)}$ is uniformly distributed in the subset $d_r$;
%the joint distribution of $(X_{(d_r)},\bd^{(d_r)})$ is given by 
%\begin{eqnarray}\label{f-per-x-p}
%      f(x_{(d_r)},\dd^{(d_r)};\dt)=
%      \prod_{u \in d_r} \left\{\frac{1}{m} f^{(u:S)}({x_{(d_r)}};\dt)\right\}^{\dd^{(d_r)}(u)},
%\end{eqnarray}
Furthermore,  by summing the  joint distribution of $(Y_{ri},\bd^{[d_r]i})$ over  $\bd^{[d_r]i}=\dd^{[d_r]i}$,    the marginal distribution of $Y_{ri}$ 
 is obtained  as follows
\[ f_{[r;m_{ri}]}(y;\dt)=     \sum_{\dd^{[d_r]i}} 
            f(y,\dd^{[d_r]i};\dt) =
             \frac{1}{m_{ri}} \sum_{v\in d_{r[r]i}} f^{[v:S]}(y;\dt).
\]
 
\end{proof}
\noindent Using Lemma \ref{le-frm}, the likelihood function under an unbalanced PROS design     is now  given by
\begin{eqnarray}\label{l-uimp}\nonumber
L(\Omega)
&=& \prod_{i=1}^{N}\prod_{r=1}^{n_i} f_{[r;m_{ri}]}(y_{ri};\dt) =
 \prod_{i=1}^{N}\prod_{r=1}^{n_i} 
\left\{
\frac{1}{m_{ri}} \sum_{v\in d_{r[r]i}} f^{[v:S]}(y_{ri};\dt)
\right\} \\
%&=& \prod_{r=1}^{n} 
%\left\{ 
%\frac{1}{m} \sum_{h=1}^{n} \sum_{u\in d_r}^{}
%\alpha_{d_r,d_h} f^{(u:S)}(x_{[d_r]};\Psi)
%\right\}  \proleft\{ 
&=& \prod_{i=1}^{N}\prod_{r=1}^{n_i} 
\left\{
\frac{1}{m_{ri}} \sum_{v\in d_{r[r]i}}
 \sum_{h=1}^{n_i} \sum_{u\in d_{h[h]i}} \frac{\alpha_{[d_r,d_h]i}}{m_{hi}} f^{(u:S)}(y_{ri};\dt)
\right\},
\end{eqnarray}
where $\Omega= (\dt, \balpha)$, $f^{(u:S)}(\cdot;\dt)$ is the pdf of the $u$-th order statistics and in a similar vein to Subsection \ref{sub:error}, $\alpha_{[d_r,d_h]i}$ is considered as the misplacement probability of a unit from subset $d_{h[h]i}$ into subset $d_{r[r]i}$ so that $\sum_{h=1}^{n_i} \alpha_{[d_r,d_h]i}=\sum_{r=1}^{n_i} \alpha_{[d_r,d_h]i} =1; i=1,\ldots,N$. Similarly, one can re-write the likelihood function \eqref{l-uimp} as follows

\begin{eqnarray*}\label{l-u2imp}
L(\Omega)
= \prod_{i=1}^{N}\prod_{r=1}^{n_i} f_{[r;m_{ri}]}(y_{ri};\dt) 
%&=& \prod_{r=1}^{n} 
%\left\{ 
%\frac{1}{m} \sum_{h=1}^{n} \sum_{u\in d_r}^{}
%\alpha_{d_r,d_h} f^{(u:S)}(x_{[d_r]};\Psi)
%\right\}  \proleft\{ 
= \prod_{i=1}^{N}\prod_{r=1}^{n_i} f(y_{ri};\dt)~ g_{ri}(y_{ri};\dt), 
\end{eqnarray*} 
where 
\begin{eqnarray}\label{u-gri}
g_{ri}(y;\dt)=  
%\frac{1}{m_{ri}} \sum_{v\in d_{r[r]i}} 
\sum_{h=1}^{n_i} \sum_{u\in d_{h[h]i}} \alpha_{[d_r,d_h]i} \frac{S}{m_{hi}}
{{S-1}\choose{u-1}} [F(y;\dt)]^{u-1} [1-F(y;\dt)]^{S-u}.
\end{eqnarray}
\noindent Similar to Subsection \ref{sub:error}, to obtain the FI matrix of an unbalanced  PROS sample and compare it with its SRS and RSS counterparts
one can easily obtain the following result.  

 %%%%%%%%%%%%%%%%%%%%%%%%%%%%%%%%%%%%%%%%%%%%%%%%%%%%%%%%%%%%%%%%%%%%%%%%%%

% latex table generated in R 3.0.2 by xtable 1.7-1 package
% Wed Jan  7 12:19:39 2015
{\begin{table}[h!]
\caption{{\footnotesize{Values of $RE_1$ and $RE_2$ to compare the FI content of unbalanced PROS data with its SRS and RSS counterparts  of the same size for normal distribution when $S=6$ and $n\in\{2,3\}$. }}}
\vspace{0.3cm} % title name of the table
\centering % centering table
\footnotesize{{\begin{tabular}{lcccccc}\hline
            &         &\multicolumn{5}{c}{$\rho$}      \\
\cline{3-7}
$D=\{d_1,\ldots,d_n\}$        & Design &0.25   &0.50   &0.75   &0.90   &1.00         \\\hline
$\{\{1,2,3,4,5\},\{6\}\}$     &  $RE_1$& 1.134 & 1.823 & 3.094 & 4.754 & 8.026 \\ [-1.7ex]
                              &  $RE_2$& 1.110 & 1.666 & 2.412 & 3.006 & 4.768  \\ 
$\{\{1,2,3,4\},\{5,6\}\}$     &  $RE_1$& 1.038 & 1.151 & 1.343 & 1.510 & 1.613  \\ [-1.7ex]
                              &  $RE_2$& 1.018 & 1.064 & 1.046 & 0.962 & 0.968  \\ 
$\{\{1,2,3\},\{4,5,6\}\}$     &  $RE_1$& 1.020 & 1.095 & 1.271 & 1.513 & 2.507  \\ [-1.7ex]
                              &  $RE_2$& 1.002 & 1.013 & 0.993 & 0.959 & 1.494  \\ 
$\{\{1,2\},\{3,4,5,6\}\}$     &  $RE_1$& 1.040 & 1.198 & 1.361 & 1.547 & 1.597  \\[-1.7ex] 
                              &  $RE_2$& 1.020 & 1.094 & 1.058 & 0.980 & 0.945  \\ 
$\{\{1\},\{2,3,4,5,6\}\}$     &  $RE_1$& 1.137 & 1.796 & 3.170 & 4.748 & 8.175  \\ [-1.7ex]
                              &  $RE_2$& 1.120 & 1.654 & 2.467 & 3.021 & 4.859  \\\hline 
$\{\{1\},\{2\},\{3,4,5,6\}\}$ &  $RE_1$& 1.071 & 1.485 & 2.196 & 2.927 & 3.389  \\ [-1.7ex]
                              &  $RE_2$& 1.052 & 1.331 & 1.599 & 1.688 & 1.374  \\ 
$\{\{1\},\{2,3\},\{4,5,6\}\}$ &  $RE_1$& 1.169 & 1.444 & 2.259 & 3.261 & 5.810  \\ [-1.7ex]
                              &  $RE_2$& 1.139 & 1.263 & 1.550 & 1.829 & 2.301  \\ 
$\{\{1\},\{2,3,4\},\{5,6\}\}$ &  $RE_1$& 1.120 & 1.385 & 2.513 & 3.620 & 5.900  \\ [-1.7ex]
                              &  $RE_2$& 1.079 & 1.228 & 1.738 & 2.063 & 2.411  \\ 
$\{\{1\},\{2,3,4,5\},\{6\}\}$ &  $RE_1$& 1.204 & 2.039 & 4.263 & 7.090 & 16.439  \\ [-1.7ex]
                              &  $RE_2$& 1.186 & 1.787 & 3.018 & 3.962 & 6.604  \\ 
$\{\{1,2\},\{3,4,5\},\{6\}\}$ &  $RE_1$& 1.038 & 1.544 & 2.484 & 3.604 & 5.734  \\ [-1.7ex]
                              &  $RE_2$& 1.004 & 1.373 & 1.761 & 2.023 & 2.278  \\ 
$\{\{1,2\},\{3,4\},\{5,6\}\}$ &  $RE_1$& 1.032 & 1.158 & 1.453 & 1.865 & 3.785  \\ [-1.7ex]
                              &  $RE_2$& 1.005 & 1.025 & 1.036 & 1.045 & 1.513  \\ 
$\{\{1,2,3\},\{4\},\{5,6\}\}$ &  $RE_1$& 0.979 & 0.923 & 0.918 & 1.129 & 2.809  \\ [-1.7ex]
                              &  $RE_2$& 0.932 & 0.813 & 0.652 & 0.642 & 1.127 \\ 
$\{\{1,2,3\},\{4,5\},\{6\}\}$ &  $RE_1$& 0.994 & 0.939 & 0.946 & 1.089 & 2.874  \\ [-1.7ex]
                              &  $RE_2$& 0.961 & 0.845 & 0.681 & 0.606 & 1.143  \\ 
$\{\{1,2,3,4\},\{5\},\{6\}\}$ &  $RE_1$& 1.086 & 1.378 & 2.178 & 2.955 & 3.463 \\ [-1.7ex]
                              &  $RE_2$& 1.077 & 1.203 & 1.553 & 1.685 & 1.386  \\ 
   \hline
\end{tabular}
\label{ta:unb6}}}
\end{table}
}

 \begin{lemma} \label{le:f-uimp}
Let $Y_{r,i}=X_{[d_r]i}$,   $r=1,\ldots, n_i; i=1,\ldots,N$, be observed from a continuous  distribution with pdf  $f(\cdot; \dt)$ using  an unbalanced PROS sampling design.  Suppose  $f_{[r;m_{ri}]}(\cdot;\dt)$ and $g_{ri}(\cdot;\dt)$ are  defined 
as in Lemma \ref{le-frm} and \eqref{u-gri}, respectively. Under the regularity conditions \BL{  of {\rm \cite{chen-bai-sinha}}}, we have

\begin{itemize}

\item [(i)] 
 $\sum_{i=1}^{N}\sum_{r=1}^{n_i}E
\left\{
\frac{D^2_{\dt}g_{ri}(X_{[d_r]i};\dt)}{g_{ri}(X_{[d_r]i};\dt)}
\right\}=
\sum_{i=1}^{N}\sum_{r=1}^{n_i} E
\left\{
{D^2_{\dt}g_{ri}(X;\dt)}
\right\},
$
\item [(ii)]
$
\sum_{i=1}^{N}\sum_{r=1}^{n_i} E
\left\{
\frac{[D_{\dt} g_{ri}(X_{[d_r]i};\dt)][D_{\dt} g_{ri}(X_{[d_r]i};\dt)]^{\top}}{g_{ri}^2(X_{[d_r]i};\dt)}
\right\}
= \BL{ \sum_{i=1}^{N}\sum_{r=1}^{n_i}}
E
\left\{
\frac{[D_{\dt} g_{ri}(X;\dt)][D_{\dt} g_{ri}(X;\dt)]^{\top}}{g_{ri}(X;\dt)}
\right\}.
$
\end {itemize}
\end{lemma}

aNow, we can present the main result of tho section as follows.

\begin{theorem}\label{th:fi-uimp}
Under the conditions of Lemma \ref{le:f-uimp}, the FI matrix of an  unbalanced PROS sample about 
unknown parameters $\Omega=(\balpha,\dt)$  is given by  
\begin{eqnarray*}\label{th:fis-imp}
\I_{upros}(\Omega)
=\I_{srs}(\dt)
- \sum_{i=1}^{N}\sum_{r=1}^{n_i} E
\left\{
{D^2_{\dt}g_{ri}(X;\dt)}
\right\}
+\sum_{i=1}^{N}\sum_{r=1}^{n_i} 
E
\left\{
\frac{[D_{\dt} g_{ri}(X;\dt)][D_{\dt} g_{ri}(X;\dt)]^{\top}}{g_{ri}(X;\dt)}
\right\}.
%&=& \I_{srs}(\dt)+\sum_{r=1}^{n} \tilde\Delta_r,
\end{eqnarray*}
%where $\sum_{r=1}^{n} \tilde\Delta_r$ is a \BL{  non-negative} definite matrix. 
%Therefore, the FI of 
%incomplete data under imperfect PROS sampling is larger than SRS data of the same size.
%\noindent{\bf Proof.} From Lemma \ref{le:f-imp}, the proof is straightforward and is omitted.
\end{theorem}
\noindent 
 Table \ref{ta:unb6} shows  the FI content of unbalanced PROS samples compared with their  SRS  and RSS counterparts in the case of normal distribution and when $N=1$, $S=6$ and three subsets $n=3$ of different sizes have been declared. The misplacement ranking error models are obtained following  the model proposed in \cite{dell1972ranked} when  $\rho\in\{0.25,0.5,0.75,0.9,1\}$.

%%%%%%%%%%%%%%%%%%%%%%%%%%%%%%%%%%%%%%%%%%%%%%%%%%%%%%%%%%%%%%%%%%%%%%%%%%%%%%%%%%%%%%%%%%

\nocite{*}

\end{document}